\documentclass[letterpaper]{article} 
\usepackage{aaai2026}  
\usepackage{times}  
\usepackage{helvet}  
\usepackage{courier}  
\usepackage[hyphens]{url}  
\usepackage{graphicx} 
\usepackage{natbib}  
\usepackage{caption} 
\usepackage{amsmath,amsfonts,amssymb,amsthm,mathtools}
\usepackage{icomma}
\usepackage{mathtools,amsthm}
\usepackage{color}
\allowdisplaybreaks[4]
\usepackage{etoolbox}
\usepackage{algorithm}
\usepackage{algpseudocode}

\usepackage{mathtools}
\usepackage{makecell}
\usepackage{nicematrix}
\usepackage{tablefootnote}
\usepackage{pdfpages}
\usepackage{soul}

\newcounter{annotatecount}
\newcounter{annotateidx}
\newcounter{annotatejdx}
\newcounter{annotatelabelcount}
\newcommand{\atran}[2]{\stepcounter{annotatecount}\overset{(\alph{annotatecount})}{#1}\csgdef{annotatedescription\theannotatecount}{#2}}
\newcommand{\aeq}[1]{\atran{=}{#1}}
\newcommand{\aleq}[1]{\atran{\leq}{#1}}

\newcommand{\annotateinitused}{\setcounter{annotateidx}{0}\whileboolexpr{test{\ifnumless{\theannotateidx}{\theannotatecount}}}{\stepcounter{annotateidx}\csgdef{aused\theannotateidx}{0}}}
\newcommand{\annotategetlabels}{\setcounter{annotatejdx}{0}\setcounter{annotatelabelcount}{0}\whileboolexpr{test{\ifnumless{\theannotatejdx}{\theannotatecount}}}{\stepcounter{annotatejdx}\ifcsequal{annotatedescription\theannotateidx}{annotatedescription\theannotatejdx}{\csgdef{aused\theannotatejdx}{1}\stepcounter{annotatelabelcount}\csedef{annotatelabel\theannotatelabelcount}{(\alph{annotatejdx})}}{}}}
\newcommand{\annotateprintlabels}{\setcounter{annotatejdx}{0}\whileboolexpr{test{\ifnumless{\theannotatejdx}{\theannotatelabelcount}}}{\stepcounter{annotatejdx}\ifnumequal{\theannotatejdx}{\theannotatelabelcount}{\ifnumequal{\theannotatejdx}{1}{}{~and~}}{}\csuse{annotatelabel\theannotatejdx}\ifnumless{\theannotatejdx}{\theannotatelabelcount}{\ifnumless{\theannotatejdx+1}{\theannotatelabelcount}{,~}{}}{}}}
\newcommand{\annotate}{\annotateinitused\setcounter{annotateidx}{0}\whileboolexpr{test{\ifnumless{\theannotateidx}{\theannotatecount}}}{\stepcounter{annotateidx}\ifcsstring{aused\theannotateidx}{0}{\ifnumequal{\theannotateidx}{1}{}{;~}\annotategetlabels\annotateprintlabels~\csuse{annotatedescription\theannotateidx}}{}}\setcounter{annotatecount}{0}}

\newcommand{\norm}[1]{\left\lVert#1\right\rVert}
\newcommand{\sqn}[1]{\left\lVert#1\right\rVert^2}

\newcommand{\texe}{\tau}

\newtheorem{theorem}{Theorem}[]
\newtheorem{definition}{Definition}

\newtheorem{lemma}{Lemma}
\newtheorem{remark}{Remark}
\newtheorem{proposition}{Proposition}
\newtheorem{corollary}{Corollary}

\newtheorem{assumption}{Assumption}

\newcommand{\E}[1]{\mathbb{E}\left[#1\right]}

\newcommand{\Ec}[2]{\mathbb{E}\left[#1\;\middle|\;#2\right]}

\newcommand{\R}{\mathbb{R}}

\newcommand{\ones}{\mathbf{1}}

\newcommand{\cE}{\mathcal{E}}
\newcommand{\cF}{\mathcal{F}}
\newcommand{\cG}{\mathcal{G}}
\newcommand{\cH}{\mathcal{H}}

\newcommand{\cL}{\mathcal{L}}

\newcommand{\cO}{\mathcal{O}}

\newcommand{\cV}{\mathcal{V}}

\newcommand{\mD}{\mathbf{D}}

\newcommand{\mg}{\mathbf{g}}

\newcommand{\mI}{\mathbf{I}}

\newcommand{\mP}{\mathbf{P}}
\newcommand{\mQ}{\mathbf{Q}}

\newcommand{\mT}{\mathbf{T}}

\newcommand{\mW}{\mathbf{W}}

\newcommand{\ol}[1]{\overline{#1}}
\newcommand{\uln}[1]{\underline{#1}}
\usepackage{newfloat}
\usepackage{listings}
\usepackage{booktabs}
\usepackage{xr}
\title{Stochastic Decentralized Optimization of Non-Smooth Convex and Convex-Concave Problems over Time-Varying Networks}

\author {
    Maxim Divilkovskiy\textsuperscript{\rm 1,2}, Alexander Gasnikov\textsuperscript{\rm 2,1,3}
}
\affiliations {
    \textsuperscript{\rm 1}MIPT\\
    \textsuperscript{\rm 2}Research Center of the Artificial Intelligence Institute, Innopolis University, Innopolis, Russia\\
    \textsuperscript{\rm 3}ISP RAS
}

\begin{document}

\maketitle

\begin{abstract}
We study non-smooth stochastic decentralized optimization problems over time-varying networks, where objective functions are distributed across nodes and network connections may intermittently appear or break. Specifically, we consider two settings: (i) stochastic non-smooth (strongly) convex optimization, and (ii) stochastic non-smooth (strongly) convex–(strongly) concave saddle point optimization. Convex problems of this type commonly arise in deep neural network training, while saddle point problems are central to machine learning tasks such as the training of generative adversarial networks (GANs). Prior works have primarily focused on the smooth setting, or time-invariant network scenarios. We extend the existing theory to the more general non-smooth and stochastic setting over time-varying networks and saddle point problems. Our analysis establishes upper bounds on both the number of stochastic oracle calls and communication rounds, matching lower bounds for both convex and saddle point optimization problems.
\end{abstract}

\begingroup
\renewcommand{\thefootnote}{}
\footnotetext{Published in Proceedings of the AAAI Conference on Artificial Intelligence (AAAI-26). \url{https://doi.org/10.1609/aaai.v40i43.41011}}
\endgroup

\section{Introduction}\label{sec:intro}
Distributed optimization is an important area in modern optimization. It has many applications in power control \citep{gan2013optimal}, vehicle control \citep{wang2010}, resource allocation \citep{beck2014allocation}, cooperative optimization \citep{Nedic2009}, and, most notably, machine learning \citep{Rabbat2004,JMLR:v11:forero10a,Galakatos2018}.
The rapid growth in the number of model parameters created a demand for running algorithms on several nodes. Another direction is machine learning with privacy constraints \citep{Ram2009}, which require separating data between servers. We study the decentralized setting of distributed optimization. In this scenario, all the nodes are equal and do not differ from each other. Another property of decentralized optimization is that communication between nodes may not be precisely scheduled.

Decentralized optimization consists in optimizing a function $f$, which can be represented as a sum of functions: $f = \sum_{i=1}^n f_i$, where each function $f_i$ is stored in a distinct node. In the decentralized time-varying setting, connections between nodes may appear or break in the process of optimization. The time-static optimization case is covered extensively in recent works \citep{Gorbunov_2022,dvinskikh2021decentralizedparallelprimaldual,Scaman2017optimal,Scaman2018optimal}, however the development of time-varying optimization started in the recent years. This setting poses more complex communication scheme than time-static one due to instability in connections.

In this research, we focus on the non-smooth stochastic formulation of convex minimization and convex-concave saddle point problems. Algorithms for saddle point problems are motivated by different modern machine learning approaches like GANs \citep{goodfellow2014,gidel2020} and reinforcement learning \citep{Jin2020,omidshafiei2017deepdecentralizedmultitaskmultiagent,wai2019multiagentreinforcementlearningdouble}. Other applications are optimal transport \citep{jambulapati2019directtildeo1epsiloniterationparallel} and economics \citep{facchinei2007finite}. Most recent research on distributed optimization, including saddle point problems assume smoothness of the considered functions \citep{Rogozin_2024,metelev2022decentralizedsaddlepointproblemsdifferent}. In this paper, we do not assume this restriction since without smoothness we can solve larger scope of problems.

The non-smooth setting of convex deterministic decentralized optimization over time-varying graphs was studied in \cite{kovalev2024lower}. In this paper, we extend their algorithm to handle arbitrary stochastic monotone operators. Thus, our contributions include establishing the first optimal convergence rates for stochastic decentralized non-smooth convex problems, as well as for both deterministic and stochastic decentralized non-smooth saddle point problems over time-varying graphs. The contributions are summarized in Table~\ref{tbl:contrib}

\begin{table}[h]
	\centering
	\caption{Summary of contributions.}
	\label{tbl:contrib}
	\begin{tabular}{lcc}
		\toprule
		\textbf{Problem} & \textbf{Deterministic} & \textbf{Stochastic} \\
		\midrule
		Convex & \cite{kovalev2024lower} & \makecell{Theorems~\ref{th:st_mntn},\ref{th:mntn}\\ (this paper)}  \\
		\midrule
		Saddle Point & \makecell{Corollary of \\Theorems~\ref{th:st_mntn},\ref{th:mntn}\\ (this paper)} & \makecell{Theorems~\ref{th:st_mntn},\ref{th:mntn}\\ (this paper)} \\
		\bottomrule
	\end{tabular}
\end{table}


\section{Problem statement}\label{sec:statement}

We consider the following two stochastic decentralized optimization problems.

\textit{Stochastic decentralized (strongly) convex non-smooth optimization:}
\begin{equation}\label{prob:cvx}
	\min_{x \in \mathbb{R}^d} \left[p(x) = \frac{1}{n}\sum_{i=1}^{n}f_i(x) + \frac{r}{2}\sqn{x}\right].
\end{equation}

\textit{Stochastic decentralized (strongly) convex-(strongly) concave non-smooth optimization:}
\begin{equation}\label{prob:spp}
	\small
	\min_{\xi \in \mathbb{R}^{d_\xi}}\max_{\zeta \in \mathbb{R}^{d_\zeta}} \left[p(\xi, \zeta) = \frac{1}{n}\sum_{i=1}^{n}f_i(\xi, \zeta) + \frac{r}{2}\sqn{\xi} - \frac{r}{2}\sqn{\zeta}\right].
\end{equation}

Throughout the paper, we define the solution space $\mathcal{H}$ as follows: $\mathcal{H} = \mathbb{R}^d$ for problem~\eqref{prob:cvx} and $\mathcal{H} = \mathbb{R}^{d_\xi + d_\zeta}$ for problem~\eqref{prob:spp}. We denote by $\norm{\cdot}$ the standard Euclidean norm in $\mathcal{H}$, and by $\langle \cdot, \cdot \rangle$ the standard inner product in $\mathcal{H}$.

For both problems \eqref{prob:cvx} and \eqref{prob:spp}, we assume $r \geq 0$. When $r > 0$, the problem is referred to as \textit{strongly monotone}; when $r = 0$, it is referred to as \textit{monotone}.

\begin{remark}
	We also study strongly convex-strongly concave problems with different constants of strong convexity and strong concavity:
	\begin{equation}\label{prob:spp_asymmetric}
		p(\xi, \zeta) = \frac{1}{n}\sum_{i=1}^{n}f_i(\xi, \zeta) + \frac{r_\xi}{2}\sqn{\xi} - \frac{r_\zeta}{2}\sqn{\zeta}.
	\end{equation}
	We obtain results for this case as a corollary of the symmetric case.
\end{remark}

To ensure well-posedness of the considered problems and to enable convergence analysis, we impose standard convexity and convexity-concavity conditions on the objective functions:

\begin{assumption}[Convexity for the problem \eqref{prob:cvx}]\label{ass:cvx_convexity}
	Each function \[ f_i(x) : \mathbb{R}^d \rightarrow \mathbb{R} \] is convex in \( x \).
\end{assumption}

\begin{assumption}[Convexity-concavity for the problem \eqref{prob:spp}]\label{ass:spp_convexity}
	Each function \[ f_i(\xi, \zeta) : \mathbb{R}^{d_\xi} \times \mathbb{R}^{d_\zeta} \rightarrow \mathbb{R} \] is convex in \( \xi \) for each fixed \( \zeta \), and concave in \( \zeta \) for each fixed \( \xi \).
\end{assumption}

\begin{remark}
	Any strongly convex or strongly convex-concave problem can be brought into the form of problems \eqref{prob:cvx} and \eqref{prob:spp_asymmetric}, respectively, by appropriately choosing the regularization parameters. In particular, any $\mu-$strongly convex function $f$ can be rewritten as $$\left( f(x) - \frac{\mu}{2} \sqn{x}\right) + \frac{\mu}{2} \sqn{x}.$$ A similar transformation applies to strongly convex–strongly concave functions.
\end{remark}

We further assume the existence of a solution for both problems.

\begin{assumption}[Existence of solution]\label{ass:radius}
	For problems \eqref{prob:cvx}, \eqref{prob:spp} there exists a solution $x^*$ such that, for some distance $R > 0,$
	\begin{equation}
		\norm{x^*} \leq R.
	\end{equation}
\end{assumption}
This assumption is crucial for non-strongly monotone problems, where a solution may not exist in general. In the case of strongly monotone problems, the solution always exists and is unique. We also require this constant $R$ for our convergence analysis.

To unify the analysis of both problems, we define operators associated with each problem.

\begin{definition}\label{def:operators}
	Let $x \in \cH$ be arbitrary. For problem \eqref{prob:cvx}, define the associated operator as
	\begin{equation}
		\mT_i(x) = \partial f_i(x).
	\end{equation}
	For problem \eqref{prob:spp}, define it as
	\begin{equation}
		\mT_i(x) = [\partial_\xi f_i(\xi, \zeta),\ -\partial_\zeta f_i(\xi, \zeta)],
	\end{equation}
	where $x = (\xi, \zeta)$.
\end{definition}

We use the notation \[\mT(x) = \left( \mT_1(x), \ldots, \mT_n(x) \right).\]

This definition allows us to treat both optimization and saddle point problems within a unified analysis. In the convex minimization case \eqref{prob:cvx}, each operator $\mT_i$ coincides with the subdifferential mapping of the corresponding convex function $f_i$, which is a set-valued monotone operator. For the saddle point problem \eqref{prob:spp}, the operator $\mT_i$ collects the partial subdifferentials with respect to the primal variable $\xi$ and the negative dual variable $\zeta$, capturing the first-order stationarity condition. Further, we assume that these operators are bounded, which is equivalent to the Lipschitz continuity of the underlying functions.

\begin{assumption}\label{ass:lipschitz}
	Let $x$ be arbitrary, and let $g_i \in \mT_i(x)$, where $\mT_i$ is defined in Definition~\ref{def:operators}. Then, for all $i \in \{1, \ldots, n\}$, 
	\begin{equation}
		\norm{g_i} \leq M.
	\end{equation}
\end{assumption}

To estimate the convergence rate for these problems, we introduce gap functions for each problem. These gap functions measure how close our result to the solution $x^*$, which exists due to the Assumption~\ref{ass:radius}.

\begin{definition}[Gap function for the problem~\eqref{prob:cvx}]\label{def:gap_cvx}
	\begin{equation}
		G_{\mathrm{CVX}}(x_o) = p(x_o) - p(x^*).
	\end{equation}
\end{definition}

\begin{definition}[Gap function for the problem~\eqref{prob:spp}]\label{def:gap_spp}
	\begin{equation}
		G_{\mathrm{SPP}}(x_o) = p(\xi_o, \zeta^*) - p(\xi^*, \zeta_o),
	\end{equation}
	where $x_o = (\xi_o, \zeta_o), x = (\xi, \zeta)$.
\end{definition}

It is well known that problem \eqref{prob:cvx} is a special case of a problem \eqref{prob:spp}. With introduced gap functions, lower bound for convex optimization will also be the lower bound for saddle point optimization. The upper bound for the saddle point optimization will also be the upper bound for convex optimization. Thus, if both bounds coincide, we can conclude that these gaps are optimal and cannot be improved.

In our convergence rate analysis, we determine the number of communications and oracle calls required to ensure that the expected gap function is bounded by $\varepsilon$. Due to the stochastic nature of the problem, we analyze the expectation of the gap function, meaning the convergence rate guarantees $\E{ G_{\mathrm{CVX}}(x_o)} \leq \varepsilon$ or $\E{G_{\mathrm{SPP}}(x_o)} \leq \varepsilon.$

\section{Stochastic decentralized setting}\label{sec:dec}

The design of deterministic decentralized algorithms for time-static networks is provided in \cite{Scaman2017optimal} for smooth problems and in \cite{Scaman2018optimal} for non-smooth problems. However, their algorithms rely on a dual oracle, which may be inaccessible in practical implementations. In contrast, \cite{kovalev2020optimal} proposed a reformulation of the decentralized problem as a Forward-Backward algorithm, achieving the optimal convergence rate in the time-static scenario while using only a primal oracle.

This idea was later extended to time-varying graphs, attaining optimal convergence rates with a deterministic primal oracle for both smooth \cite{Kovalev2021Lower} and non-smooth \cite{kovalev2024lower} settings.

In our paper, we follow this reformulation, extending the analysis to stochastic primal settings and saddle-point problems.

We start by formalizing the time-varying optimization setting. At a fixed moment of time, a communication network may be represented as a graph $\cG(\cV, \cE),$ where $\cV = \{1, \ldots, n\}$ is a set of nodes and $\cE \subseteq \cV \times \cV$ is a set of edges between these nodes. In our scenario, the connections may change over time. Therefore, we represent the time-varying communication network as a function of a continuous time parameter $\tau > 0$ as a function $\cG(\tau) = (\cV, \cE(\tau))$, where $\cE(\tau): \R_+ \rightarrow 2^{\cV \times \cV}$ denotes the time-varying set of edges.

Next, we formalize the mechanism of node interaction, which is commonly modeled through \textit{gossip matrix} multiplication. In the time-static setting, the gossip matrix remains constant throughout the execution of the algorithm and corresponds to the fixed structure of the underlying communication network. In contrast, the time-varying setting poses additional challenges, as the network topology evolves over time, making it inappropriate to associate a single fixed matrix with the entire process.

We represent the gossip matrix as a matrix-valued function $$\mW(\tau): \R_+ \rightarrow \R^{n \times n},$$ which satisfies Assumption~\ref{ass:gossip}. See Figure~\ref{fig:time-varying-graph} for the example of a time-varying graph.

\begin{figure}[!htpb]
	\centering
	\includegraphics[width=\linewidth]{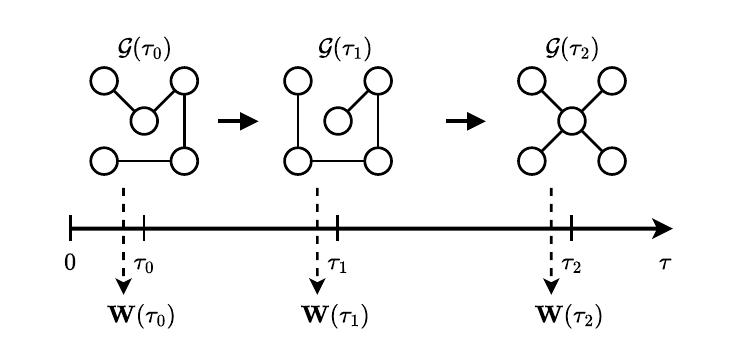}
	\caption{Illustration of a time-varying communication graph. At each moment in time, the associated gossip matrix reflects the current network configuration.}
	\label{fig:time-varying-graph}
\end{figure}

\begin{assumption}\label{ass:gossip}
	For all $\tau \geq 0$, the gossip matrix $\mW(\texe) \in \R^{n\times n}$ associated with the time-varying communication network $\cG(\cV, \cE(\texe))$ satisfies the following properties:
	\begin{enumerate}
		\item $\mW(\texe)_{ij} = 0$ if $i\neq j$ and $(j,i) \notin \cE(\texe)$,
		\item $\mW(\texe) \ones_n = 0$ and $\mW(\texe)^\top \ones_n = 0$.
	\end{enumerate}
\end{assumption}

The first condition encodes the structure of the network. The second implies that the gossip step converges to the average over the whole network. A common choice for the gossip matrix is a Laplacian matrix of a graph $\cG(\tau)$.

For the convergence analysis, we also introduce a condition number $\chi$ for the time-varying network as follows.

\begin{assumption}\label{ass:chi}
	There exists a constant $\chi \geq 1$ such that the following inequality holds for all $\texe \geq 0$:
	\begin{multline}
		\sqn{\mW(\texe) x - x} \leq \left(1 - \frac{1}{\chi}\right)\sqn{x}
		\;\;\\
		\text{for all}
		\;\;
		x \in \left\{x \in \R^n : {\sum_{i=1}^{n}}x_i = 0\right\}.
	\end{multline}
\end{assumption}

The constant $\chi \geq 1$ quantifies the connectivity of the time-varying network. A larger value of $\chi$ indicates poorer connectivity and results in more iterations required for the algorithm to converge.

Next, we define the stochastic decentralized oracle. In the deterministic setting, this coincides with the standard definition of a decentralized first-order oracle. Stochasticity arises from allowing some deviation from the true operator value. We define a first-order oracle for both problems simultaneously by treating it as an oracle that returns a stochastic approximation of the corresponding operator.
\begin{definition}[Stochastic decentralized oracle]
	For arbitrary $x$, let $\mT_i$ be defined as in Definition~\ref{def:operators} for the problems \eqref{prob:cvx} and \eqref{prob:spp}. 
	A random vector $\tilde{\mg}_i(x): \mathcal{H} \rightarrow \mathcal{H}$ is called a stochastic operator oracle associated with operator $\mT_i(x)$ if there exists $\sigma > 0$ such that for any $x \in \mathcal{H}$ the following holds:
	\begin{align}
		& \E{\tilde{\mg}_i(x)} = \mg_i \in \mT_i;\\
		& \E{\sqn{\tilde{\mg}_i(x) - \mg_i}_2} \leq \sigma^2.
	\end{align}
	Here, $\mT_i(x)$ is the true subdifferential, $\mg_i(x)$ is the subgradient and $\tilde{\mg_i}(x)$ is a stochastic subgradient.
\end{definition}

\section{Decentralized reformulation and the Algorithm}\label{sec:algorithm}

\begin{algorithm*}
	\caption{}
	\label{alg}
	\begin{algorithmic}[1]
		\State {\bf input:} $x^0 = x^{-1} = \tilde{x}^0\in \cH^n$, $y^0 = \bar{y}^0\in \cH^n$, $z^0 = \bar{z}^0\in \cL^\perp$, $m^0 \in \cH^n$
		\State {\bf parameters:}
		$K,T \in \{1,2,\ldots\}$,\\ $\{(\alpha_k,\beta_k,\gamma_k,\sigma_k,\lambda_k,\tau_x^k,\eta_x^k,\eta_y^k,\eta_z^k, \theta_z^k)\}_{k=0}^{K-1}\subset \R_+^{10}$
		\For{$k = 0,1,\ldots,K-1$}
		\State $\uln{y}^k = \alpha_k y^k + (1-\alpha_k)\ol{y}^k$,\quad
		$\uln{z}^k = \alpha_k z^k + (1-\alpha_k)\ol{z}^k$
		\label{line:comb}
		\State $g_y^k = \nabla_y G(\uln{y}^k,\uln{z}^k)$,\quad
		$g_z^k = \nabla_z G(\uln{y}^k,\uln{z}^k)$\quad
		\label{line:grad}
		\State $\tilde{g}_z^k = (\mW(\tau) \otimes\mI_d)g_z^k$,\;\;
		$\hat{g}_z^k =  (\mW(\tau) \otimes\mI_d)(g_z^k + m^k)$
		\label{line:grad_comm}
		\Comment{Decentralized communication}
		\State $y^{k+1} = y^k - \eta_y^k (g_y^k + \hat{x}^{k+1})$,\quad
		$z^{k+1} = z^k - \eta_z^k \hat{g}_z^k$,\quad
		$\hat{x}^{k+1} = x^k + \gamma_k (\tilde{x}^k - x^{k-1})$\label{line:yz}
		\State $\ol{y}^{k+1} = \uln{y}^k + \alpha_k(y^{k+1} - y^k)$,\quad
		$\ol{z}^{k+1} = \uln{z}^k - \theta_z^k \tilde{g}_z^k$,\quad
		$m^{k+1} = (\eta_z^k/\eta_z^{k+1})(m^k +g_z^k  - \hat{g}_z^k)$
		\label{line:ext}
		\State $x^{k,0} = x^k$ \label{line:iinit}
		\For{$t = 0,1,\ldots,T-1$}
		\State $g_x^{k,t} = (\tilde{\mg}_1(x_1^{k,t}),\ldots,\tilde{\mg}_n(x_n^{k,t}))$ \label{line:ig}
		\Comment{Stochastic oracle call}
		\State $x^{k,t+1} = x^{k,t} - \eta_x^k \left(g_x^{k,t} + \beta_k x^{k,t+1} - y^{k+1} + \tau_x^k(x^{k,t+1} - x^k) \right)$ \label{line:ix}
		\EndFor
		\State $x^{k+1} = \sigma_k x^{k,T} + (1-\sigma_k)\tilde{x}^{k+1}$,\quad
		$\tilde{x}^{k+1} = \frac{1}{T}\sum_{t=1}^{T} x^{k,t}$,\quad
		$\ol{x}^{k+1} = \alpha_k \tilde{x}^{k+1} + (1-\alpha_k)\ol{x}^k$
		\label{line:ox}
		\EndFor
		\State $(x_a^K,y_a^K,z_a^K) = (\sum_{k=1}^K \lambda_k)^{-1}\sum_{k=1}^K \lambda_k(\ol{x}^k,\ol{y}^k,\ol{z}^k)$
		\label{line:avg}
		\State {\bf output:} $x_o^K = \frac{1}{n}\sum_{i=1}^{n}x_{a,i}^K \in \cH$,\quad
		where\;$(x_{a,1}^K,\ldots,x_{a,n}^K) = x_a^K \in \cH^n$
		\label{line:output}
	\end{algorithmic}
\end{algorithm*}

A reformulation of the convex problem in the deterministic setting was presented in~\cite{kovalev2024lower}. In this paper, we follow a similar approach but generalize the algorithm to arbitrary monotone operators. Specifically, we extend the algorithm to handle any stochastic monotone operator instead of relying on subgradients of convex functions. Unlike mentioned paper, which is based on a saddle point reformulation of the convex minimization problem, we bypass this step by directly reformulating both problems as a monotone inclusion problems.

First, we perform a standard distributed reformulation. Specifically, denote the consensus space \begin{equation}
	\cL = \{(x_1, \ldots, x_n) \in \cH^n : x_1 = \ldots = x_n\}.
\end{equation}

In the analysis we will also need the fact that \begin{equation}
	\cL^\perp = \{(x_1, \ldots, x_n) \in \cH^n : \sum_{i=1}^n x_i = 0\}.
\end{equation}

Hence, the problem \eqref{prob:cvx} is equivalent to the following problem:
\begin{equation}\label{prob:dec_cvx}
	\min_{x \in \cH} \left[\frac{1}{n}\sum_{i=1}^{n}f_i(x_i) + \frac{r}{2n}\sqn{x}\right] \text{subject to}\quad x \in \cL.
\end{equation}

The problem \eqref{prob:spp} is equivalent to the following problem:
\begin{equation}\label{prob:dec_spp}
	\begin{aligned}
		\min_{\xi \in \mathbb{R}^{d_\xi}} \max_{\zeta \in \mathbb{R}^{d_\zeta}} 
		& \Bigg[ \frac{1}{n} \sum_{i=1}^{n} f_i(\xi_i, \zeta_i) 
		+ \frac{r}{2n} \|\xi\|^2 
		- \frac{r}{2n} \|\zeta\|^2 \Bigg] \\
		& \text{subject to } (\xi, \zeta) = x \in \mathcal{L}.
	\end{aligned}
\end{equation}
Now, we incorporate consensus via communication into the optimization problem. Let $\mT_i$ be defined as in Definition~\ref{def:operators}. Define
\begin{gather}
	\mT(x) = \begin{bmatrix} \mT_1(x_1)\\\vdots\\\mT_n(x_n) \end{bmatrix} + r_x x : \mathcal{H}^n \rightarrow (2^{\mathcal{H}})^n;\\
	G(y, z) = \frac{r_{yz}}{2}\sqn{y+z}: \cH^n \times \cH^n \rightarrow \cH,
\end{gather}
where,
\begin{align*}
	x &= (x_1, \ldots, x_n) \in \mathcal{H}^n, \\
	\text{and}\quad & r_x, r_{yz} > 0 \ \text{satisfy}\ r_x + \frac{1}{r_{yz}} = r.
\end{align*}

Let $\mathsf{E} = \cH^n \times \cH^n \times \cL^\perp$ be an Euclidean space.

Consider the operators
\[A(u) = \begin{bmatrix} 0 \\ \nabla_y G(y,z) \\ \mP \nabla_z G(y,z) \end{bmatrix}; \quad B(u) = \begin{bmatrix} \mT(x) - y \\ x \\ 0 \end{bmatrix},\]
where $$\mP = (\mI_n - (1/n)\ones_n\ones_n^\top)\otimes \mI_d$$ for the problem \eqref{prob:cvx} and $$\mP = (\mI_n - (1/n)\ones_n\ones_n^\top)\otimes \mI_{d_\xi + d_\zeta}$$ for the problem \eqref{prob:spp}. $\mP$ is the orthogonal projection matrix onto $\cL^\perp$. Then, operator $A$ is a monotone operator and corresponds to the gradient of a smooth convex function. Operator $B$ is a monotone set-valued operator.

Consider the following monotone inclusion problem:
\begin{equation}\label{eq:prob_mon}
	\text{find $u \in \mathsf{E}$ such that } 0 \in A(u) + B(u).
\end{equation}

This problem can be solved using the Forward-Backward Algorithm with Nesterov acceleration (see \cite{kovalev2020optimal} for a similar approach). The acceleration relies on the fact that the operator $A$ is a gradient of a smooth function. Since the operator $\mT$ is not a gradient of a smooth function we have to place it into the operator $B$. This reformulation is a key to the presented algorithm. The following theorem establishes the equivalence of this reformulation.

\begin{theorem}\label{th:reformulation}
	Problem \eqref{eq:prob_mon} is equivalent to problems \eqref{prob:cvx} and \eqref{prob:spp} with the corresponding definitions of $\mT_i$ for each of the problems.
\end{theorem}

The proof of this theorem is provided in the Supplementary Materials in Section~\ref{sec:proof1}.

Multiplication by the matrix $\mP$ corresponds to projecting onto the consensus space. This, in turn, means averaging values across the entire network, which is challenging in the time-varying setting due to changing connectivity. To address this, the algorithm replaces global averaging via $\mP$ with local averaging using a matrix $\mW$, where each multiplication by this matrix corresponds to averaging over the immediate neighbors of each node in the network graph.

As mentioned earlier, a similar reformulation was introduced in \cite{Kovalev2021Lower}. However, in their setting, the functions were smooth, and thus the operator $\mT$ corresponded to the gradient of a smooth function. When incorporated into the operator $A$, convergence could be accelerated using Nesterov's acceleration. Based on this, they proposed the ADOM algorithm, which achieves the optimal convergence rate for their setting.

In our case, the operator $\mT$ is incorporated into $B$, while iterations over operator $A$ are accelerated using Nesterov's method. The iterations involving operator $B$ cannot be further accelerated, either for convex or saddle-point problems. This aligns with classical results in non-smooth, non-distributed optimization.

With this setup, we show that Algorithm~\ref{alg} converges to the desired solution. The algorithm introduces an additional input variable $m$, corresponding to the error-feedback mechanism. The $y$ and $z$ updates are accelerated using Nesterov’s acceleration.

The inner loop over $T$ corresponds to gradient descent for problem~\eqref{prob:cvx}, and to gradient descent–ascent for problem~\eqref{prob:spp}. The algorithm requires $K$ decentralized communication rounds and $K \times T$ stochastic oracle calls.

\section{Optimal convergence rate}\label{sec:optimal}

In this section we assume that Assumptions~\ref{ass:radius} to \ref{ass:chi} hold. For the problem \eqref{prob:cvx} Assumption~\ref{ass:cvx_convexity} holds, for the problem~\eqref{prob:spp} Assumption~\ref{ass:spp_convexity} holds.

In~\cite{kovalev2024lower}, the authors provide a lower bound for the deterministic non-smooth decentralized convex minimization problem over time-varying networks. By using their analysis and combining it with the classical lower bound for non-smooth convex optimization in the non-distributed setting \citep{bubeck2015convexoptimizationalgorithmscomplexity}, we obtain the following lower bounds.

For concise formulation, we denote the optimality gap \( G \) as
\[
G :=
\begin{cases}
	G_{\mathrm{CVX}}, & \text{for the convex minimization problem \eqref{prob:cvx}}, \\
	G_{\mathrm{SPP}}, & \text{for the saddle-point problem \eqref{prob:spp}}.
\end{cases}
\]

\begin{proposition}[Lower bound for problems \eqref{prob:cvx} and \eqref{prob:spp} in the strongly monotone case]\label{prop:st_mntn}
	Let $r > 0$. Then, for arbitrary $\varepsilon > 0$ there exists an optimization problem and a time-varying network such that Algorithm~\ref{alg} requires at least
	\begin{align*}
		\Omega\left(\frac{\chi M}{\sqrt{r \varepsilon}}\right) \text{ decentralized communications}
	\end{align*}
	and
	\begin{align*}
		\Omega\left(\frac{(M + \sigma)^2}{r \varepsilon}\right) \text{ oracle calls}
	\end{align*}
	to achieve $\E{G(x_o^K)} \leq \varepsilon$.
\end{proposition}

\begin{proposition}[Lower bound for problems \eqref{prob:cvx} and \eqref{prob:spp} in the monotone case]\label{prop:mntn}
	Let $r = 0$. Then, for arbitrary $\varepsilon > 0$ there exists an optimization problem and a time-varying network such that Algorithm~\ref{alg} requires at least
	\begin{align*}
		\Omega\left(\frac{\chi MR}{\varepsilon}\right) \text{ decentralized communications}
	\end{align*}
	and
	\begin{align*}
		\Omega\left(\frac{(M + \sigma)^2 R^2}{\varepsilon^2}\right) \text{ oracle calls}
	\end{align*}
	to achieve $\E{G(x_o^K)} \leq \varepsilon$.
\end{proposition}

These lower bounds match the corresponding lower bounds for the non-distributed setting as well as for deterministic decentralized non-smooth convex optimization.

We now present the following theorems on the convergence rate of Algorithm~\ref{alg}.

\begin{theorem}[Upper bound for problems~\eqref{prob:cvx} and \eqref{prob:spp} in the strongly monotone case]\label{th:st_mntn}
	Let $r > 0$. Then, for arbitrary $\varepsilon > 0$ Algorithm~\ref{alg}, requires 
	\begin{align*}
		\cO\left(\frac{\chi M}{\sqrt{r \varepsilon}}\right) \text{ decentralized communications}
	\end{align*}
	and
	\begin{align*}
		\cO \left(\frac{(M + \sigma)^2}{r \varepsilon}\right) \text{ oracle calls}
	\end{align*}
	to achieve $\E{G(x_o^K)} \leq \varepsilon$.
\end{theorem}
\begin{theorem}[Upper bound for problems \eqref{prob:cvx} and \eqref{prob:spp} in the monotone case]\label{th:mntn}
	Let $r = 0$. Then, for arbitrary $\varepsilon > 0$ Algorithm~\ref{alg}, requires 
	\begin{align*}
		\cO\left(\frac{\chi MR}{\varepsilon}\right) \text{ decentralized communications}
	\end{align*}
	and
	\begin{align*}
		\cO\left(\frac{(M + \sigma)^2 R^2}{\varepsilon^2}\right) \text{ oracle calls}
	\end{align*}
	to achieve $\E{G(x_o^K)} \leq \varepsilon$.
\end{theorem}
The proofs of these theorems are provided in the Supplementary Materials in Sections~\ref{sec:proof_st_mntn} and~\ref{sec:proof_mntn}.

These upper bounds match the corresponding lower bounds in Propositions~\ref{prop:st_mntn} and \ref{prop:mntn}, thus establishing the optimality of these convergence rates.

In the case of saddle-point problems, the strong convexity and strong concavity constants may differ. The following result addresses this asymmetric setting.

\begin{corollary}[Complexity for saddle point problems with different constants of strong convexity and strong concavity]\label{cor:asym}
	Consider the problem of type \eqref{prob:spp_asymmetric}. Let $r_\xi > 0$, $r_\zeta > 0$. Then, for arbitrary $\varepsilon > 0$, Algorithm~\ref{alg} requires 
	\begin{align*}
		\cO\left(\frac{\chi M}{\sqrt{\varepsilon}} \sqrt{\frac{1}{r_\xi} + \frac{1}{r_\zeta}} \right) \text{ decentralized communications}
	\end{align*}
	and
	\begin{align*}
		\cO \left(\frac{(M + \sigma)^2 }{\varepsilon} \left(\frac{1}{r_\xi} + \frac{1}{r_\zeta}\right) \right) \text{ oracle calls},
	\end{align*} to achieve $\E{G_{\mathrm{SPP}}(x_o^K)} \leq \varepsilon$.
\end{corollary}

The proof of this corollary is provided in the Supplementary Materials in Section~\ref{sec:proof_asym}. 

\section{Experiments}\label{sec:experiments}

We conduct experiments on synthetic random graphs constructed using the Erdős–Rényi model. We consider both time-static and time-varying settings. Our evaluation focuses on two metrics: the Euclidean distance to the objective, \begin{align*}
	\norm{x_o^K - x^*},
\end{align*} and the saddle point residual, \begin{align*}
	f(\xi_o^K, \zeta^*) - f(\xi^*, \zeta_o^K).
\end{align*}

We begin with the random graph setup. In the time-static scenario, the graph remains fixed throughout the algorithm's execution. In contrast, the time-varying scenario involves randomly removing and adding edges at each iteration. Starting graph is shown in Figure~\ref{fig:graph}.

\begin{figure}[!htpb]
	\centering
	\includegraphics[width=0.4\linewidth]{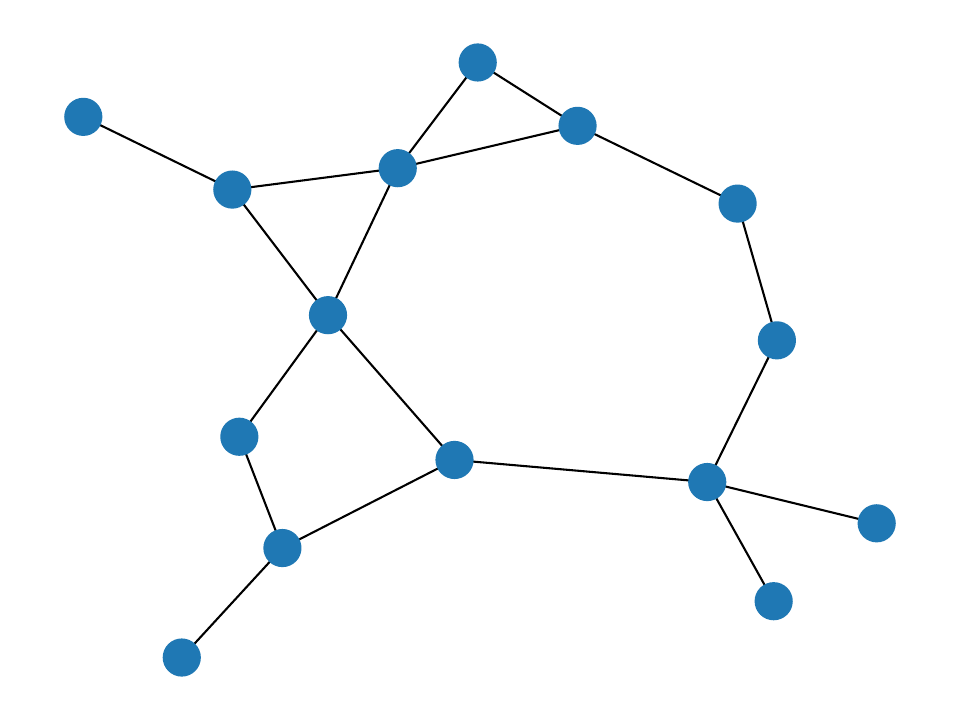}
	\caption{Example of a random graph with 15 nodes.}
	\label{fig:graph}
\end{figure}

We test our method on a simple non-smooth saddle-point problem of the form:
\begin{align*}
	\min_{\xi \in \mathbb{R}^{d_\xi}} \max_{\zeta \in \mathbb{R}^{d_\zeta}} p(\xi, \zeta) = \sum_{i=1}^n f_i(\xi, \zeta) + \frac{r}{2} \|\xi\| - \frac{r}{2} \|\zeta\|,
\end{align*}
where $r = 10^{-3}$ and
\begin{align*}
	f_i(\xi, \zeta) = \|\xi - c_{\xi, i}\|_1 - \|\zeta - c_{\zeta, i}\|_1,
\end{align*}
with $c_{\xi, i}$ and $c_{\zeta, i}$ some fixed constants.

These functions are convex in \( \xi \) and concave in \( \zeta \). The subgradients with respect to \( \xi \) and \( \zeta \) are computed separately.

Subgradient with respect to $\xi$:
\begin{align*}
	\partial_\xi f_i(\xi, \zeta) &= \partial \|\xi - c_{\xi, i}\|_1 \\
	&= \left\{ v \in \mathbb{R}^{d_\xi} \;\middle|\; v_j \in 
	\begin{cases}
		\{\operatorname{sign}(\xi_j - c_{\xi, i, j})\}, \\ 
		\quad \xi_j \neq c_{\xi, i, j}, \\
		[-1, 1], \\ 
		\quad \xi_j = c_{\xi, i, j}
	\end{cases} 
	\right\}.
\end{align*}

Subgradient with respect to $\zeta$:
\begin{align*}
	\partial_\zeta f_i(\xi, \zeta) &= -\partial \|\zeta - c_{\zeta, i}\|_1 \\
	&= \left\{ -v \in \mathbb{R}^{d_\zeta} \;\middle|\; v_j \in 
	\begin{cases}
		\{\operatorname{sign}(\zeta_j - c_{\zeta, i, j})\}, \\
		\quad \zeta_j \neq c_{\zeta, i, j}, \\
		[-1, 1], \\
		\quad \zeta_j = c_{\zeta, i, j}
	\end{cases} 
	\right\}.
\end{align*}

Then, we can write operator $\mT_i$ from the Definition~\ref{def:operators} of $f_i(\xi, \zeta)$ as a pair:
\[
\mT_i(\xi, \zeta) = \left( \partial_\xi f_i(\xi, \zeta), \; -\partial_\zeta f_i(\xi, \zeta) \right).
\]

We evaluate the performance for $K \in \{1, \ldots, 30\}$, with a fixed parameter $T = 10$.

\begin{figure}[!htpb]
	\centering
	\includegraphics[width=0.9\linewidth]{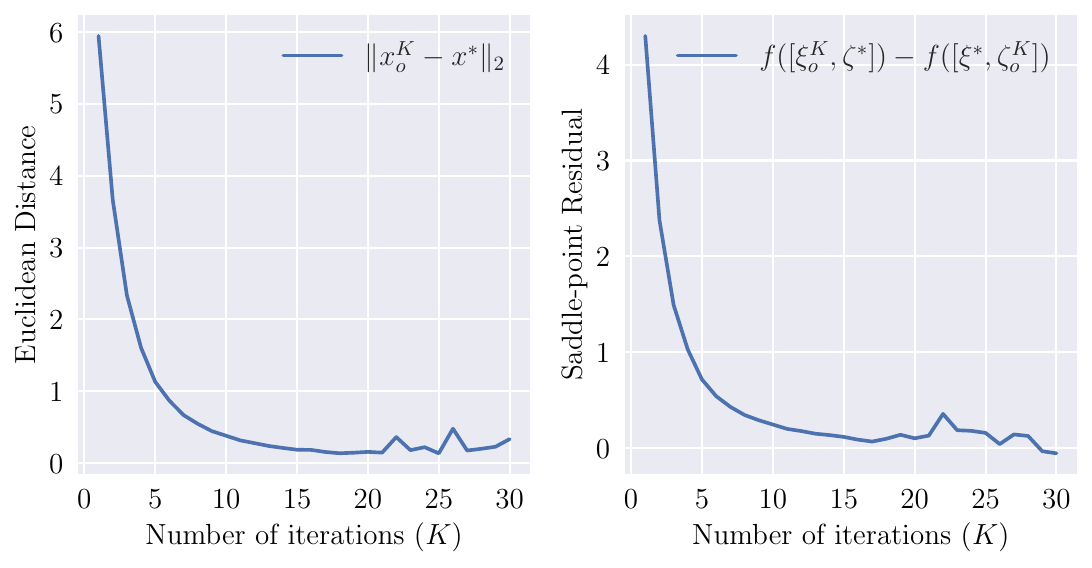}
	\caption{Performance over time-varying graphs.}
	\label{fig:time-static}
\end{figure}

\begin{figure}[!htpb]
	\centering
	\includegraphics[width=0.9\linewidth]{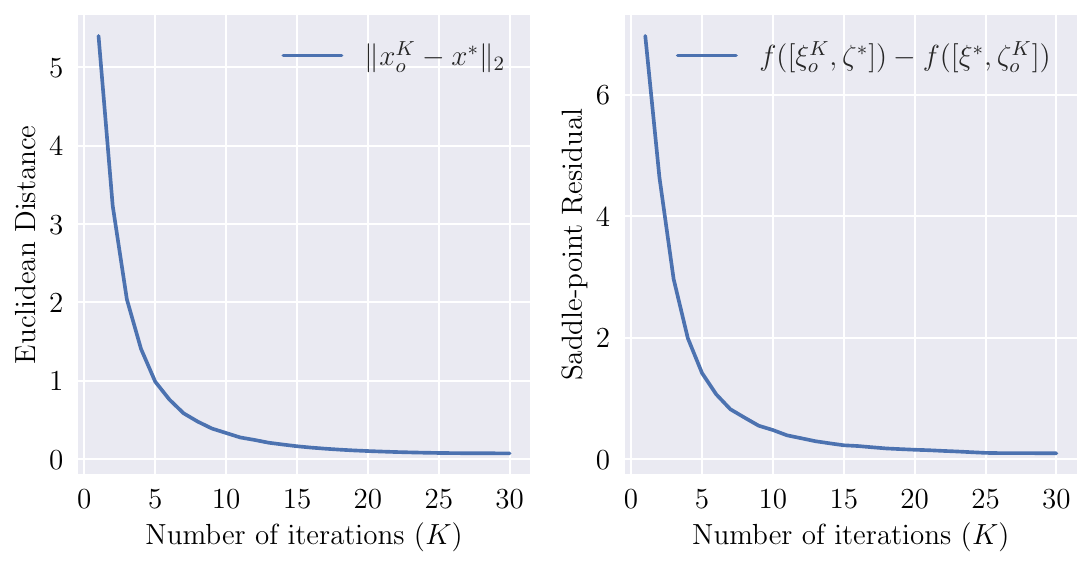}
	\caption{Performance over time-static graphs.}
	\label{fig:time-varying}
\end{figure}

Figure~\ref{fig:time-static} presents the performance of the method under the time-varying graph setting. As observed, the randomness in the graph structure introduces some fluctuations in both performance metrics. Conversely, Figure~\ref{fig:time-varying} illustrates the results under the time-static setting, where the graph remains fixed. In this case, we observe a more stable and monotonic decrease in error across iterations.

\section{Conclusion}\label{sec:conclusion}

In this paper, we investigate non-smooth stochastic decentralized optimization over time-varying networks, addressing both convex minimization and convex-concave saddle point problems. Our research extends previous work that focused primarily on smooth settings or time-static networks; thus we provide a more general result. We study both monotone and strongly monotone scenarios. We consider both monotone (weakly convex/concave) and strongly monotone problems, as well as the asymmetric case where the strong convexity and strong concavity parameters differ.

We generalize the deterministic algorithm from \cite{kovalev2024lower} to handle arbitrary stochastic monotone operators, making it applicable to a broader class of real-world problems. Our main theorems establish optimal convergence rates (matching theoretical lower bounds) for both convex minimization and saddle point problems. Moreover, these rates are optimal for the deterministic case and the non-distributed case as well.

A possible future direction may be studying the case of asymmetric oracle cost. In saddle point problems, the computational cost of querying the convex (primal) and concave (dual) oracles may differ (e.g., in GANs). Extending our analysis could result in a better convergence rate with this assumption.

\section*{Acknowledgements}
The study was supported by the Ministry of Economic Development of the Russian Federation (agreement No. 139-10-2025-034 dd. 19.06.2025, IGK000000C313925P4D0002)

\appendix
\onecolumn

\setcounter{equation}{0}
\renewcommand{\theequation}{S\arabic{equation}}

\section{Proof of Theorem 1}\label{sec:proof1}
We start by showing that the solution to the problem~\eqref{eq:prob_mon} is a solution to the original problems. From $0 \in A(u) + B(u),$ we have
\begin{gather*}
		0 \in \mT(x) - y;\\
		0 = r_{yz} (y + z) + x;\\
		0 = \mP r_{yz} (y + z),
\end{gather*}
rearranging, we get
\begin{gather*}
		y \in \mT(x);\\
		y = \frac{-x}{r_{yz}} - z;\\
		r_{yz} (y+z) \in \cL,
\end{gather*}
substituting $y$ into first row, using the definition of $\mT$ and using $x = -r_{yz}(y+z)$ we get
\begin{gather*}
	-z \in [\mT_i(x_i)]_{i=1}^n + r_x x + \frac{x}{r_{yz}};\\
	x \in \cL,
\end{gather*}
using $r = r_x + 1/r_{yz}$ we obtain
\begin{gather*}
		-z \in [\mT_i(x_i)]_{i=1}^n + r x;\\
		x \in \cL.
\end{gather*}

From this point, the proof splits depending on whether the problem is convex or saddle point optimization.

Convex case:

Summing the first inclusion over \( i = 1, \ldots, n \) and using \( z \in \cL^\perp \), \( x \in \cL \), we obtain
\begin{align*}
	0 \in \sum_{i=1}^n \partial f_i(x_1) + n r x_1 = n \partial p(x_1).
\end{align*}

Hence,
\begin{align*}
	0 \in \partial p(x_1).
\end{align*}
Thus, $x_1$ is a solution to the problem~\eqref{prob:cvx} by optimality condition for convex functions.

Saddle point case:

Similarly, summing over $i = 1 \ldots n$ yields
\begin{align*}
	0 \in \begin{pmatrix}
		\partial_{\xi} f_i(\xi_1, \zeta_1)\\
		-\partial_{\zeta} f_i(\xi_1, \zeta_1)
	\end{pmatrix} + n r \begin{pmatrix}
		\xi_1\\\zeta_1
	\end{pmatrix}.
\end{align*}
Thus, combining both parts for $\xi$ and $\zeta$ we obtain the optimality condition for convex-concave functions.

To see that the solution $x$ of the original problems solves the problem~\eqref{eq:prob_mon} it is sufficient to take any $z \in \cL$ and $y = \frac{-x}{r_{yz}} - z$, which concludes the proof. \hfill $\square$

\section{Auxiliary lemmas}\label{sec:lemmas}
In this section, we focus on proving auxiliary lemmas for Theorems~\ref{th:st_mntn} and~\ref{th:mntn} for the saddle point optimization problem. Note that the upper bounds established automatically imply the corresponding upper bounds for the convex optimization case, due to the structure of the gap functions and the problems structure.

We assume that Assumptions~\ref{ass:spp_convexity} to~\ref{ass:chi} hold. For the convergence analysis of Algorithm~\ref{alg}, we require the following auxiliary lemmas and definitions.

\begin{definition}
	For the problem~\ref{prob:spp} define
	\begin{equation}
		F(\xi, \zeta) = \sum_{i=1}^{n} f_i(\xi_i, \zeta_i) + \frac{r_x}{2} \sqn{\xi} - \frac{r_x}{2} \sqn{\zeta};
	\end{equation}
	\begin{equation}
		\mD(x_o, x) = F(\xi_o, \zeta) - F(\xi, \zeta_o),
	\end{equation}
	where $x = (\xi, \zeta), x_o = (\xi_o, \zeta_o).$
	Note that $\mD(x_o, x)$ is convex in $x_o$ and concave in $x$.
	
	Also, define
	\begin{equation}
		\mQ(x, y, z, x_o, y_o, z_o) = \mD(x_o, x) - \langle y, x_o \rangle + \langle y_o, x \rangle - G(y, z) + G(y_o, z_o).
	\end{equation}
\end{definition}

We also set the parameters for the Algorithm~\ref{alg}. Firstly, we take
\begin{align}
	&r_x = \frac{2}{3}r,\quad r_{yz} = \frac{3}{r};\label{eq:rx}\\
	&\tau_x = \frac{1}{2} r_x,\quad \eta_y = (4 r_{yz})^{-1},\quad \eta_z = (10 r_{yz} \chi^2)^{-1}\label{eq:tau_eta}.
\end{align}

Next, for $k \in \left\{0 \ldots K - 1\right\}$:
\begin{align}
	&\alpha_k = \frac{3}{k + 3},\quad \gamma_k = \frac{k+2}{k+3};\label{eq:alpha_gamma}\\
	&\tau_x^k = \tau_x \alpha_k^{-1},\quad \eta_y^k = \eta_y \alpha_k^{-1},\quad \eta_z^k = \eta_z \alpha_k^{-1},\quad \eta_x^k = \frac{1}{\tau_x^K T};\label{eq:tau_eta_k}\\
	&\beta_k = r_x,\quad \sigma_k = \frac{\tau_x^k}{2 \tau_x^k + \beta_k},\quad \theta_z^k = \frac{1}{2r_{yz}}.\label{eq:beta_sigma_theta}\\
\end{align}

Set the parameters $\lambda_k$:
\begin{align}\label{eq:lambda}
	&\lambda_k = \alpha_{k-1}^{-2} + \alpha_k^{-1} - \alpha_k^{-2} \text{ for $k = 1, \ldots, K-1$},\quad \lambda_K = \alpha_{K-1}^{-2}.
\end{align}

Also, set the input to the Algorithm~\ref{alg}:
\begin{align}\label{eq:input}
	x^0 = 0,\quad y^0 = 0,\quad z^0 = 0,\quad m^0 = 0.
\end{align}

\begin{lemma}\label{lemm:subgrad}
	Let $r > 0$. For any $x \in \cH$, $k \in \{0, \ldots, K\}$ and $t \in \{0, \ldots, T-1\}$ the following inequality holds:
	\begin{align*}
		&\Ec{\langle \beta_k{x^{k, t+1}, x - x^{k, t+1}} \rangle}{\cF_{k,t}} - \Ec{\langle g^{k,t}_x, x^{k,t+1} - x \rangle}{\cF_{k,t}} \leq \\
		&\Ec{-\mD(x^{k, t+1}, x)}{\cF_{k,t}} - \frac{r_x}{2} \Ec{\sqn{x - x^{k, t+1}}}{\cF_{k,t}}\\
		&+ n\eta_x^k (3M + \sigma)^2/2 + \frac{1}{2\eta_x^k}\Ec{\sqn{x^{k,t} - x^{k,t+1}}}{\cF_{k,t}},
	\end{align*}
	where \(\cF_{k,t}\) is a sigma-algebra representing the history of all random variables generated by Algorithm~\ref{alg} up to inner iteration \(t\) of outer iteration \(k\):
	\[
	\cF_{k,t} := \sigma\Big(
		x^{k,0}, \ldots, x^{k,t},\;
		g_x^{k,0}, \ldots, g_x^{k,t-1},\;
		y^0,\ldots,y^{k+1},\;
		z^0,\ldots,z^{k+1},\;
		m^0,\ldots,m^{k+1}
	\Big).
	\]
\end{lemma}

\begin{proof}
	We start by upper bounding the following term:
	\begin{align*}
		& -\langle g^{k,t}_x, x^{k,t+1} - x \rangle \\
		&= -\langle \tilde{g}(x^{k,t}), x^{k,t+1} - x \rangle \\
		&= \langle \tilde{g}(x^{k,t}), x - x^{k,t+1} \rangle \\
		&= \sum_{i=1}^n \langle \tilde{g}_i(x^{k,t}), x_i - x_i^{k,t+1} \rangle \\
		&= \sum_{i=1}^n \langle g_i(x_i^{k,t}) + \omega(x_i^{k,t}), x_i - x_i^{k,t+1} \rangle \\
		&= \sum_{i=1}^n \langle g_i(x_i^{k,t}), x_i - x_i^{k,t+1} \rangle + \langle g_i(x_i^{k,t}), x_i^{k,t+1} - x_i^{k, t} \rangle \\
		&\quad + \langle \omega(x_i^{k,t}), x_i - x_i^{k,t} \rangle + \langle \omega(x_i^{k,t}), x_i^{k,t+1} - x_i^{k,t} \rangle \\
		&\aleq{uses upper bounds on $\omega$ and $g_i$} \sum_{i=1}^n \langle g_i(x_i^{k,t}), x_i - x_i^{k,t} \rangle + (M + \sigma) \norm{x_i^{k,t} - x_i^{k,t+1}} \\
		&\quad + \langle \omega(x_i^{k,t}), x_i^{k,t} - x \rangle,
	\end{align*}
	where \annotate.
	
	Note that $\Ec{\langle w(x_i^{k, t}), x_i^{k, t} - x \rangle}{\cF_{k,t}} = 0$. Then, taking expectation conditioned on $\cF_{k,t}$:
	\begin{align*}
		\mathbb{E}& \left[ - \langle g^{k,t}_x, x^{k,t+1} - x \rangle \,\middle|\, \cF_{k,t} \right]
		\leq \sum_{i=1}^n \langle g_i(x_i^{k,t}), x_i - x_i^{k,t} \rangle + (M + \sigma) \, \mathbb{E} \left[ \|x_i^{k,t} - x_i^{k,t+1}\| \,\middle|\, \cF_{k,t} \right]\\
		&\aleq{uses subgradient inequality for convex-concave functions}\sum_{i=1}^n (f_i(\xi, \zeta_i^{k, t}) - f_i(\xi_i^{k, t}, \zeta)) + (M + \sigma) \, \mathbb{E} \left[ \|x_i^{k,t} - x_i^{k,t+1}\| \,\middle|\, \cF_{k,t} \right],\\
		&\aleq{uses Lipschitz continuity of $f_i$} \sum_{i=1}^n \left(f_i(\xi, \zeta_i^{k, t+1}) - f_i(\xi_i^{k, t+1}, \zeta)\right) + M \, \mathbb{E} \left[ \|\xi_i^{k,t} - \xi_i^{k,t+1}\| +\|\zeta_i^{k,t} - \zeta_i^{k,t+1}\| \,\middle|\, \cF_{k,t} \right] \\
		&\quad + \sigma \, \mathbb{E} \left[ \|x_i^{k,t} - x_i^{k,t+1}\| \,\middle|\, \cF_{k,t} \right], \\
		&\aleq{uses the property of Euclidean norm}\sum_{i=1}^n \Ec{f_i(\xi, \zeta_i^{k, t+1}) - f_i(\xi_i^{k, t+1}, \zeta)}{\cF_{k,t}} + (M + \sqrt{2}M + \sigma) \, \mathbb{E} \left[ \|x_i^{k,t} - x_i^{k,t+1}\| \,\middle|\, \cF_{k,t} \right]\\
		&\leq\sum_{i=1}^n \Ec{f_i(\xi, \zeta_i^{k, t+1}) - f_i(\xi_i^{k, t+1}, \zeta)}{\cF_{k,t}} + (3M + \sigma) \, \mathbb{E} \left[ \|x_i^{k,t} - x_i^{k,t+1}\| \,\middle|\, \cF_{k,t} \right],
	\end{align*}
	where \annotate.
	
	Adding the term $\Ec{\langle \beta_k{x^{k, t+1}, x - x^{k, t+1}} \rangle}{\cF_{k,t}}$ to the both sides we get
	\begin{align*}
		&\Ec{\langle \beta_k{x^{k, t+1}, x - x^{k, t+1}} \rangle}{\cF_{k,t}} - \Ec{\langle g^{k,t}_x, x^{k,t+1} - x \rangle}{\cF_{k,t}}\\
		&\leq \Ec{F(\xi, \zeta^{k, t+1}) - F(\xi^{k, t+1}, \zeta)}{\cF_{k,t}} - \frac{r_x}{2}\Ec{\sqn{\xi} - \sqn{\zeta^{k, t+1}} - \sqn{\xi^{k, t+1}} + \sqn{\zeta}}{\cF_{k,t}}\\
		&+ \Ec{\langle \beta_k{x^{k, t+1}, x - x^{k, t+1}} \rangle}{\cF_{k,t}} + (3M + \sigma) \sum_{i=1}^n \Ec{\norm{x_i^{k,t} - x_i^{k,t+1}}}{\cF_{k,t}}\\
		&\aleq{uses the parallelogram rule and the definition of $\beta_k$}\Ec{F(\xi, \zeta^{k, t+1}) - F(\xi^{k, t+1}, \zeta)}{\cF_{k,t}} - \frac{r_x}{2}\Ec{\sqn{x} - \sqn{x^{k, t+1}}}{\cF_{k,t}}\\
		&+ \frac{r_x}{2}\Ec{-\sqn{x^{k, t+1}} - \sqn{x^{k, t+1} - x} + \sqn{x}}{\cF_{k,t}}\\
		&+ (3M + \sigma) \sum_{i=1}^n \Ec{\norm{x_i^{k,t} - x_i^{k,t+1}}}{\cF_{k,t}}\\
		&=\Ec{F(\xi, \zeta^{k, t+1}) - F(\xi^{k, t+1}, \zeta)}{\cF_{k,t}} - \frac{r_x}{2} \Ec{\sqn{x - x^{k, t+1}}}{\cF_{k,t}}\\
		&+ (3M + \sigma) \sum_{i=1}^n \Ec{\norm{x_i^{k,t} - x_i^{k,t+1}}}{\cF_{k,t}},
	\end{align*}
	where \annotate.
	
	Using the definition of $\mD(x_o, x)$ we obtain
	\begin{align*}
		&\Ec{\langle \beta_k{x^{k, t+1}, x - x^{k, t+1}} \rangle}{\cF_{k,t}} - \Ec{\langle g^{k,t}_x, x^{k,t+1} - x \rangle}{\cF_{k,t}}\\
		&\Ec{-\mD(x^{k, t+1}, x)}{\cF_{k,t}} - \frac{r_x}{2} \Ec{\sqn{x - x^{k, t+1}}}{\cF_{k,t}}\\
		&+ (3M + \sigma) \sum_{i=1}^n \Ec{\norm{x_i^{k,t} - x_i^{k,t+1}}}{\cF_{k,t}}\\
		&=\Ec{-\mD(x^{k, t+1}, x)}{\cF_{k,t}} - \frac{r_x}{2} \Ec{\sqn{x - x^{k, t+1}}}{\cF_{k,t}}\\
		&+ \sum_{i=1}^n (3M + \sigma) \Ec{\norm{x_i^{k,t} - x_i^{k,t+1}}}{\cF_{k,t}}\\
		&=\Ec{-\mD(x^{k, t+1}, x)}{\cF_{k,t}} - \frac{r_x}{2} \Ec{\sqn{x - x^{k, t+1}}}{\cF_{k,t}}\\
		&+ \sum_{i=1}^n \eta_x^k \sqrt{\eta_x^k}(3M + \sigma) \frac{1}{\sqrt{\eta_x^k}}\Ec{\norm{x_i^{k,t} - x_i^{k,t+1}}}{\cF_{k,t}}\\
		&\aleq{uses the Young's inequality}\Ec{-\mD(x^{k, t+1}, x)}{\cF_{k,t}} - \frac{r_x}{2} \Ec{\sqn{x - x^{k, t+1}}}{\cF_{k,t}}\\
		&+ \sum_{i=1}^n \eta_x^k (3M + \sigma)^2/2 + \frac{\Ec{\norm{x_i^{k,t} - x_i^{k,t+1}}}{\cF_{k,t}}^2}{2\eta_x^k}\\
		&\aleq{uses the Jensen's inequality and the property of Euclidean norm}\Ec{-\mD(x^{k, t+1}, x)}{\cF_{k,t}} - \frac{r_x}{2} \Ec{\sqn{x - x^{k, t+1}}}{\cF_{k,t}}\\
		&+ n\eta_x^k (3M + \sigma)^2/2 + \frac{1}{2\eta_x^k}\Ec{\sqn{x^{k,t} - x^{k,t+1}}}{\cF_{k,t}},
	\end{align*}
	where \annotate.
	This concludes the proof.
\end{proof}

\begin{lemma}\label{lemm:optimum}
	Let $r > 0$. Then the function $\mQ(x, y, z, x_o, y_o, z_o)$ is convex in $x_o, y_o, z_o$ and concave in $x, y, z$. Moreover, for a fixed solution $x^*$ of the saddle point problem \eqref{prob:spp}, there exist $w^*, y^*, z^*$, such that the following conditions hold:
	\begin{gather}
		0 \in \partial_x \mQ(w^*, y^*, z^*, w^*, y^*, z^*),\quad 0 \in \partial_{x_o} \mQ(w^*, y^*, z^*, w^*, y^*, z^*);\\
		0 = \nabla_y \mQ(w^*, y^*, z^*, w^*, y^*, z^*),\quad 0 = \nabla_{y_o} \mQ(w^*, y^*, z^*, w^*, y^*, z^*);\\
		\nabla_z \mQ(w^*, y^*, z^*, w^*, y^*, z^*) \in \cL,\quad \nabla_{z_o} \mQ(w^*, y^*, z^*, w^*, y^*, z^*)\in \cL;\\
		\sqn{w^*} \leq \frac{2nM^2}{r^2}, \quad \sqn{y^*} \leq 2\left(1 + \frac{r_x}{r}\right)^2nM^2, \quad \sqn{z^*} \leq 8nM^2.
	\end{gather}
\end{lemma}

\begin{proof}
	$\mD(x_o, x)$ is convex in $x_o$, $G(y_o, z_o)$ is convex in $y_o, z_o$, $\langle y_o, x \rangle$ and $\langle y, x_o \rangle$ are affine and thus both convex and concave. Hence, $\mQ$ is convex in $x_o, y_o, z_o$ and concave in $x, y, z$.
	
	Take $x^* = (\xi^*, \zeta^*) \in \cH$, which is the unique solution to the problem \eqref{prob:spp} since $r > 0$. Define $w^* = \left(x^*, \ldots, x^*\right) \in \cL$.
	\begin{align*}
		0 \in \partial p(x^*) = \frac{1}{n}\sum_{i = 1}^{n} \partial f_i(x^*) + r\xi^* - r\zeta^*.
	\end{align*}
	
	Hence, there exists a vector $\Delta^*$, such that $\Delta_i^* \in \partial f_i(x^*)$ and
	\begin{equation}\label{eq:xstar}
		r\xi^* - r\zeta^* + \frac{1}{n} \sum_{i = 1}^{n} \Delta_i^* = 0.
	\end{equation}
	
	Decompose each $\Delta_i^* = (\Delta_i^{\xi, *}, \Delta_i^{\zeta, *})$, where
	\begin{equation*}
		\Delta_i^{\xi, *} \in \partial_{\xi} f_i(x^*) \quad \text{and} \quad \Delta_i^{\zeta, *} \in \partial_{\zeta} f_i(x^*).
	\end{equation*}
	
	Define $y^* = (y^*_1, \ldots, y^*_n) \in \cH^n$, where for each $i$
	\begin{equation*}
		y_i^* = \begin{pmatrix}\Delta_i^{\xi, *} + r_x \xi^*\\ -\Delta_i^{\zeta, *} + r_x \zeta^*\end{pmatrix}.
	\end{equation*}
	
	Define $z^* = (z^*_1, \ldots, z^*_n)$, where for each $i$
	\begin{equation*}
		z_i^* = \begin{pmatrix}-\Delta_i^{\xi, *} - r \xi^*\\ \Delta_i^{\zeta, *} - r \zeta^*\end{pmatrix}.
	\end{equation*}
	
	From the definition of $\mD$ we have
	\begin{align*}
		&\partial_{x_o} \mD(w^*, w^*) = (
		\partial_{\xi_o} F(w^*),
		-\partial_{\zeta_o} F(w^*))\\
		&= 
		\begin{pmatrix}
			(\partial_{\xi_{o, 1}}f_1(x^*) + r_x \xi^*,
			-\partial_{\zeta_{o, 1}} f_1(x^*) + r_x \zeta^*)\\
			\vdots\\
			(\partial_{\xi_{o, n}}f_n(x^*) + r_x \xi^*,
			-\partial_{\zeta_{o, n}} f_n(x^*) + r_x \zeta^*)
		\end{pmatrix}.
	\end{align*}
	
	Hence, $y^* \in \partial_{x_o} \mD(w^*, w^*)$. Then,
	\begin{align*}
		&0 \in \partial_{x_o} (\mD(w^*, w^*) - \langle y^*, \cdot \rangle);\\
		&\partial_{x_o} \mQ(w^*, y^*, z^*, w^*, y^*, z^*) = \partial_{x_o} \mD(x_o, w^*) - y^*;\\
		&0 \in \partial_{x_o} \mQ(w^*, y^*, z^*, w^*, y^*, z^*).
	\end{align*}
	
	Note that $z^* \in \cL^\perp$.
	
	We have
	\begin{gather*}
		\nabla_{z_o} \mQ(w^*, y^*, z^*, w^*, y^*, z^*) = \nabla_{z_o} G(y^*, z^*) = r_{yz} (y^* + z^*).
	\end{gather*}

	Examining the $i-$th component, we have
	\begin{gather*}
		r_{yz}\begin{pmatrix}
			\Delta_i^{\xi, *} + r_x \xi^* -\Delta_i^{\xi, *} - r \xi^*\\
			-\Delta_i^{\zeta, *} + r_x \zeta^* + \Delta_i^{\zeta, *} - r \zeta^*
		\end{pmatrix} = r_{yz}\begin{pmatrix}
			(r_x - r) \xi^*\\
			(r_x - r) \zeta^*
		\end{pmatrix} = r_{yz}(r_x - r) x^*.
	\end{gather*}

	Hence,
	$\nabla_{z_o} \mQ(w^*, y^*, z^*, w^*, y^*, z^*) = -w^* \in \cL.$
	
	Similarly, we obtain
	\begin{gather*}
		\nabla_{y_o} \mQ(w^*, y^*, z^*, w^*, y^*, z^*) = \nabla_{y_o} G(y^*, z^*) + w^* = -w^* + w^* = 0.
	\end{gather*}
	
	By symmetry and similar derivations, the same conditions hold for the variables $x, y, z$.
	
	To obtain norm bounds we start by $\norm{\Delta_i} \leq M$ from Assumption~\ref{ass:lipschitz}. From~\eqref{eq:xstar} we get that
	\begin{align*}
		\norm{\xi^*} \leq \frac{M}{r};\\
		\norm{\zeta^*} \leq \frac{M}{r}.
	\end{align*}

	Combining, $\sqn{x^*} = \sqn{\xi^*} + \sqn{\zeta^*} \leq \frac{2M^2}{r^2}$. Hence,
	\begin{equation*}
		\sqn{w^*} \leq \frac{2nM^2}{r^2}.
	\end{equation*}
	
	Similarly from the definitions of $y^*$ and $z^*$ we get
	\begin{equation*}
		\sqn{y^*} \leq 2(1 + r_x/r)^2nM^2,
	\end{equation*}
	and
	\begin{equation*}
		\sqn{z^*} \leq 8nM^2,
	\end{equation*}
	which concludes the proof.
\end{proof}

\begin{lemma}\label{lemm:iteration}
	Let $r > 0$. For any $x \in \cH$ and $k \in \{0, \ldots, K-1\}$ the following inequality holds:
	\begin{align*}
		(\tau_x^k + \frac{1}{2} r_x) &\E{\sqn{x^{k+1} - x}}\\
		&\leq \tau_x^k \E{\sqn{x^k - x}} - \E{\mD(\tilde{x}^{k+1}, x)} + \E{\langle y^{k+1}, \tilde{x}^{k+1} - x \rangle}\\
		&- \frac{\tau_x^k}{2} \E{\sqn{\tilde{x}^{k+1} - x^k}} + \frac{n (3M + \sigma)^2}{2 \tau_x^k T}.
	\end{align*}
\end{lemma}

\begin{proof}
	\begin{align*}
		\frac{1}{2\eta_x^k} ||&x^{k, t+1} - x||^2 \\
		&\aeq{uses the parallelogram rule}      \frac{1}{2\eta_x^k} ||x^{k, t} - x||^2 - \frac{1}{2\eta_x^k}||x^{k, t+1}-x^{k,t}||^2 - \beta_k\langle x^{k, t+1}, x^{k, t+1} - x \rangle\\
		&- \tau_x^k \langle x^{k, t+1} - x^{k}, x^{k, t+1} - x \rangle + 
		\langle y^{k+1}, x^{k, t+1} - x \rangle - \langle g_x^{k,t}, x^{k,t+1} - x \rangle \\
		& \aleq{uses Cauchy-Schwarz inequality}  \frac{1}{2\eta_x^k} ||x^{k, t} - x||^2 - \frac{1}{2\eta_x^k}||x^{k, t+1}-x^{k,t}||^2 - \frac{\tau_x^k}{2} ||x^{k, t+1} - x^k||^2 - \frac{\tau_x^k}{2} ||x^{k, t+1} - x||^2
		\\&+ \frac{\tau_x^k}{2} ||x^k - x||^2 + 
		\langle y^{k+1}, x^{k, t+1} - x \rangle - \beta_k\langle x^{k, t+1}, x^{k, t+1} - x \rangle - \langle g_x^{k,t}, x^{k,t+1} - x \rangle,
	\end{align*}
	where \annotate.
	
	Define
	\begin{align*}
		S_{t} := \frac{1}{2\eta_x^k} ||x^{k, t} - x||^2 + \frac{\tau_x^k}{2} ||x^k - x||^2
	\end{align*}

	Thus,
	\begin{align*}
		\frac{1}{2\eta_x^k} ||&x^{k, t+1} - x||^2 \leq S_t - \frac{\tau_x^k}{2} ||x^{k, t+1} - x^k||^2 \\&- \frac{\tau_x^k}{2} ||x^{k, t+1} - x||^2 + \langle y^{k+1}, x^{k, t+1} - x \rangle  - \frac{1}{2\eta_x^k}||x^{k, t+1}-x^{k,t}||^2\\
		&+ \beta_k \langle x^{k, t+1}, x - x^{k, t+1} \rangle - \langle g_x^{k,t}, x^{k,t+1} - x \rangle,
	\end{align*}
	
	Taking 	\[
		\cF_{k,t} := \sigma\Big(
			x^{k,0}, \ldots, x^{k,t},\;
			g_x^{k,0}, \ldots, g_x^{k,t-1},\;
			y^0,\ldots,y^{k+1},\;
			z^0,\ldots,z^{k+1},\;
			m^0,\ldots,m^{k+1}
		\Big),
		\]
	and applying conditional expectation $\Ec{\cdot}{\cF_{k, t}}$ and using Lemma~\ref{lemm:subgrad} we get
	\begin{align*}
		\frac{1}{2\eta_x^k} &\Ec{\sqn{x^{k, t+1} - x}}{\cF_{k,t}} \leq S_t + \Ec{- \frac{\tau_x^k}{2} ||x^{k, t+1} - x^k||^2 - \frac{\tau_x^k}{2} ||x^{k, t+1} - x||^2}{\cF_{k,t}} \\&+ \Ec{\langle y^{k+1}, x^{k, t+1} - x \rangle}{\cF_{k,t}} - \Ec{\mD(x^{k,t+1}, x)}{\cF_{k,t}}\\
		&- \frac{r_x}{2} \Ec{\sqn{x - x^{k, t+1}}}{\cF_{k,t}}
		+ n\eta_x^k (3M + \sigma)^2/2\\
		&\leq \Ec{-\frac{\tau_x^k}{2} \sqn{x^{k, t+1} - x^{k}} - \frac{\tau_x^k + r_x}{2} \sqn{x^{k, t+1} - x}}{\cF_{k,t}} \\
		&+ \Ec{\langle y^{k+1}, x^{k, t+1} - x \rangle - \mD(x^{k, t+1}, x)}{\cF_{k,t}}\\ 
		&+ n\eta_x^k (3M + \sigma)^2/2 + S_t.
	\end{align*}
	
	Summing inequalities for $t = 0, \ldots, T-1$ we get
	\begin{align*}
		&\sum_{i=1}^{T} \frac{1}{2\eta_x^k} \Ec{\sqn{x^{k, i} - x}}{\cF_{k,i-1}}\\
		&\leq \sum_{i=0}^{T-1} \frac{1}{2\eta_x^k} \Ec{\sqn{x^{k, i} - x}}{\cF_{k,i-1}} + \frac{T \tau_x^k}{2} \sqn{x^k - x}\\
		&+\sum_{i=1}^{T} \Ec{-\frac{\tau_x^k}{2} \sqn{x^{k, i} - x^{k}} - \frac{\tau_x^k + r_x}{2} \sqn{x^{k, i} - x}}{\cF_{k,i-1}}\\
		&+ \Ec{\langle y^{k+1}, x^{k, i} - x \rangle - \mD(x^{k, i}, x)}{\cF_{k,i-1}}\\
		&+T n\eta_x^k (3M + \sigma)^2/2.
	\end{align*}
	
	Now, applying expectation to both sides and using the fact that $\E{\Ec{Z}{\cdot}} = \E{Z}$ we get
	\begin{align*}
		&\sum_{i=1}^{T} \frac{1}{2\eta_x^k} \E{\|x^{k,i} - x\|^2} \leq
		\sum_{i=0}^{T-1} \frac{1}{2\eta_x^k} \E{\|x^{k,i} - x\|^2}
		+ \frac{T \tau_x^k}{2} \|x^k - x\|^2 \\
		& + \sum_{i=1}^{T} \E{
			-\frac{\tau_x^k}{2} \|x^{k,i} - x^k\|^2
			-\frac{\tau_x^k + r_x}{2} \|x^{k,i} - x\|^2
			+ \langle y^{k+1}, x^{k,i} - x \rangle
			- \mD(x^{k,i}, x)
		}\\
		&+ \frac{T n \eta_x^k}{2} (3M + \sigma)^2.
	\end{align*}
	
	Then, using $x^{k, 0} = x^k$ and convexity of $\mD$ we get
	\begin{align*}
		&\frac{1}{2\eta_x^k T} \E{\sqn{x^{k,T} - x}} \leq \frac{1}{2\eta_x^k T} \E{\sqn{x^k - x}} + \frac{\tau_x^k}{2} \sqn{x^k - x} + \frac{n \eta_x^k}{2} (3M + \sigma)^2\\
		&+ \E{
			-\frac{\tau_x^k}{2} \sqn{\tilde{x}^{k+1} - x^k}
			-\frac{\tau_x^k + r_x}{2} \sqn{\tilde{x}^{k+1} - x}
			+ \langle y^{k+1}, \tilde{x}^{k+1} - x \rangle
			- \mD(\tilde{x}^{k+1}, x)
		}.
	\end{align*}
	
	Using the definition of $\eta_x^k$ and $x^{k+1}$ we get
	\begin{align*}
		&(\tau_x^k + \frac{1}{2} r_x) \E{\sqn{x^{k+1} - x}}\\
		& \leq \tau_x^k \E{\sqn{x^k - x}} + \frac{n (3M + \sigma)^2}{2 \tau_x^k T} - \E{\mD(\tilde{x}^{k+1}, x)} + \E{\langle y^{k+1}, \tilde{x}^{k+1} - x \rangle}\\ &- \frac{\tau_x^k}{2} \E{\sqn{\tilde{x}^{k+1} - x^k}},
	\end{align*}
	which concludes the proof.
	
\end{proof}

\begin{lemma}\label{lemm:gap_iteration}
	Let $r > 0$. For any $x, y \in \cH$ and $z \in \cL^\perp$ the following holds:
	\begin{align}
		\E{\mQ(x, y, z, x_a^K, y_a^K, z_a^K)} &\leq \frac{1}{K^2} \left(2r \sqn{x} + \frac{36}{r} \sqn{y} + \frac{90 \chi^2}{r} \sqn{z} \right) + \frac{18n(3M+\sigma)^2}{rKT}.
	\end{align}
\end{lemma}

\begin{proof}
	From Lemma 11 of the paper \citep{kovalev2024lower}, we get
	\begin{align*}
		&\frac{1}{2\eta_y}\sqn{y^{k+1} - y}
		+\frac{1}{2\eta_z}\sqn{\hat{z}^{k+1} - z}
		+\frac{1}{2\eta_z}\sqn{\eta_z^{k+1} m^{k+1}}_{\mP}
		\\&\leq
		\frac{1}{2\eta_y}\sqn{y^{k} - y}
		+\frac{1}{2\eta_z}\sqn{\hat{z}^{k} - z}
		+\frac{1}{2\eta_z}\sqn{\eta_z^k m^k}_{\mP}
		+2\eta_y\alpha_k^{-2}\gamma_k^2\sqn{x^{k-1} - \tilde{x}^k}
		\\&
		+\gamma_k\alpha_k^{-1} \langle x^{k-1} - \tilde{x}^k, y^k - y \rangle
		-\alpha_k^{-1}\langle x^k - \tilde{x}^{k+1}, y^{k+1} - y\rangle
		-\alpha_k^{-1}\langle \tilde{x}^{k+1}, y^{k+1} - y\rangle
		\\&
		+(\alpha_k^{-2}-\alpha_k^{-1})\left(G(\ol{y}^k,\ol{z}^k) - G(y,z)\right)
		-\alpha_k^{-2}\left(
		G(\ol{y}^{k+1},\ol{z}^{k+1})-G(y,z)
		\right).
	\end{align*}
	
	From Lemma~\ref{lemm:iteration} we get
	\begin{align*}
		&(\tau_x^k + \frac{1}{2} r_x) \E{\sqn{x^{k+1} - x}}\\
		& \leq \tau_x^k \E{\sqn{x^k - x}} + \frac{n (3M + \sigma)^2}{2 \tau_x^k T} - \E{\mD(\tilde{x}^{k+1}, x)}\\
		&+ \E{\langle y^{k+1}, \tilde{x}^{k+1} - x \rangle} - \frac{\tau_x^k}{2} \sqn{\tilde{x}^{k+1} - x^k}.
	\end{align*}
	
	Dividing this inequality by $\alpha_k$ and conditioning on \[\cF_k = \sigma(x^k, \ldots, x^0, y^k, \ldots, y^0, z^k, \ldots, z^0),\] we get
	\begin{align*}
		& \tau_x (\alpha_k^{-2} + \alpha_k^{-1}) \Ec{\sqn{x^{k+1} - x}}{\cF_k} \leq \tau_x \alpha_k^{-2} \sqn{x^k - x}\\
		& + \frac{n (3M + \sigma)^2}{2 \tau_x T} - \alpha_k^{-1}\Ec{\mD(\tilde{x}^{k+1}, x)}{\cF_k}\\
		&+ \alpha_k^{-1} \Ec{\langle y^{k+1}, \tilde{x}^{k+1} - x \rangle}{\cF_k} - \frac{\tau_x \alpha_k^{-2}}{2} \Ec{\sqn{\tilde{x}^{k+1} - x^k}}{\cF_k}.
	\end{align*}

	Combining with the previous inequality conditioned on $\cF_k$ we get
	\begin{align*}
		&\tau_x (\alpha_k^{-2} + \alpha_k^{-1}) \Ec{\sqn{x^{k+1} - x}}{\cF_k}
		+ \frac{1}{2\eta_y}\sqn{y^{k+1} - y}\\
		&+\frac{1}{2\eta_z}\sqn{\hat{z}^{k+1} - z}
		+\frac{1}{2\eta_z}\sqn{\eta_z^{k+1} m^{k+1}}_{\mP}
		\\&\leq
		\tau_x \alpha_k^{-2} \sqn{x^k - x}
		+\frac{1}{2\eta_y}\sqn{y^{k} - y}\\
		&+\frac{1}{2\eta_z}\sqn{\hat{z}^{k} - z}
		+\frac{1}{2\eta_z}\sqn{\eta_z^k m^k}_{\mP}
		+\frac{n (3M + \sigma)^2}{2 \tau_x T}
		\\&
		-\frac{\tau_x \alpha_k^{-2}}{2} \Ec{\sqn{\tilde{x}^{k+1} - x^k}}{\cF_k}
		+2\eta_y\alpha_k^{-2}\gamma_k^2\sqn{x^{k-1} - \tilde{x}^k}
		\\&
		+\gamma_k\alpha_k^{-1} \langle x^{k-1} - \tilde{x}^k, y^k - y \rangle
		-\alpha_k^{-1} \Ec{\langle x^k - \tilde{x}^{k+1}, y^{k+1} - y\rangle}{\cF_k}
		\\&
		- \alpha_k^{-1}\left(\Ec{\mD(\tilde{x}^{k+1}, x) + \langle y^{k+1}, x \rangle - \langle \tilde{x}^{k+1}, y \rangle \right)}{\cF_k}
		\\&
		+(\alpha_k^{-2}-\alpha_k^{-1})\left(G(\ol{y}^k,\ol{z}^k) - G(y,z)\right)
		-\alpha_k^{-2}\left(
		G(\ol{y}^{k+1},\ol{z}^{k+1})-G(y,z)
		\right).
	\end{align*}
	
	Using the definition of $\mQ$ we get
	\begin{align*}
		&\tau_x (\alpha_k^{-2} + \alpha_k^{-1}) \Ec{\sqn{x^{k+1} - x}}{\cF_k}
		+ \frac{1}{2\eta_y}\sqn{y^{k+1} - y}\\
		&+\frac{1}{2\eta_z}\sqn{\hat{z}^{k+1} - z}
		+\frac{1}{2\eta_z}\sqn{\eta_z^{k+1} m^{k+1}}_{\mP}
		\\&\leq
		\tau_x \alpha_k^{-2} \sqn{x^k - x}
		+\frac{1}{2\eta_y}\sqn{y^{k} - y}
		+\frac{1}{2\eta_z}\sqn{\hat{z}^{k} - z}\\
		&+\frac{1}{2\eta_z}\sqn{\eta_z^k m^k}_{\mP}
		+\frac{n (3M + \sigma)^2}{2 \tau_x T}
		\\&
		-\frac{\tau_x \alpha_k^{-2}}{2} \Ec{\sqn{\tilde{x}^{k+1} - x^k}}{\cF_k}
		+2\eta_y\alpha_k^{-2}\gamma_k^2\sqn{x^{k-1} - \tilde{x}^k}
		\\&
		+\gamma_k\alpha_k^{-1} \langle x^{k-1} - \tilde{x}^k, y^k - y \rangle
		-\alpha_k^{-1} \Ec{\langle x^k - \tilde{x}^{k+1}, y^{k+1} - y\rangle}{\cF_k}
		\\&
		+(\alpha_k^{-2} - \alpha_k^{-1}) \mQ(x, y, z, \bar{x}^k, \bar{y}^k, \bar{z}^k)
		-\alpha_k^{-2} \Ec{\mQ(x, y, z, \bar{x}^{k+1}, \bar{y}^{k+1}, \bar{z}^{k+1})}{\cF_k}.
	\end{align*}
	
	Take $\alpha_k$ as in Equation~\ref{eq:alpha_gamma}. Then, \[\alpha_k^{-2} + \alpha_k^{-1} \geq \alpha_{k+1}^{-2},\quad \gamma_k \alpha_k^{-1} = \frac{k+2}{3},\quad \alpha_k^{-1} = \frac{k+3}{3}.\]
	
	Hence, we get
	\begin{align*}
		&\tau_x \alpha_{k+1}^{-2} \Ec{\sqn{x^{k+1} - x}}{\cF_k}
		+ \frac{1}{2\eta_y}\sqn{y^{k+1} - y}
		+\frac{1}{2\eta_z}\sqn{\hat{z}^{k+1} - z}
		+\frac{1}{2\eta_z}\sqn{\eta_z^{k+1} m^{k+1}}_{\mP}
		\\
		&\leq
		\tau_x \alpha_k^{-2} \sqn{x^k - x}
		+\frac{1}{2\eta_y}\sqn{y^{k} - y}
		+\frac{1}{2\eta_z}\sqn{\hat{z}^{k} - z}\\
		&+\frac{1}{2\eta_z}\sqn{\eta_z^k m^k}_{\mP}
		+\frac{n (3M + \sigma)^2}{2 \tau_x T}
		\\&
		-\frac{\tau_x (k+3)^2}{18} \Ec{\sqn{\tilde{x}^{k+1} - x^k}}{\cF_k}
		+2\eta_y \frac{(k+2)^2}{9} \sqn{x^{k-1} - \tilde{x}^k}
		\\&
		+\frac{k+2}{3} \langle x^{k-1} - \tilde{x}^k, y^k - y \rangle
		-\frac{k+3}{3} \Ec{\langle x^k - \tilde{x}^{k+1}, y^{k+1} - y\rangle}{\cF_k}
		\\&
		+(\alpha_k^{-2} - \alpha_k^{-1}) \mQ(x, y, z, \bar{x}^k, \bar{y}^k, \bar{z}^k)
		-\alpha_k^{-2} \Ec{\mQ(x, y, z, \bar{x}^{k+1}, \bar{y}^{k+1}, \bar{z}^{k+1})}{\cF_k}\\
		&\aeq{uses the definition of $\eta_y$ and $\tau_x$}
		\tau_x \alpha_k^{-2} \sqn{x^k - x}
		+\frac{1}{2\eta_y}\sqn{y^{k} - y}
		+\frac{1}{2\eta_z}\sqn{\hat{z}^{k} - z}\\
		&+\frac{1}{2\eta_z}\sqn{\eta_z^k m^k}_{\mP}
		+\frac{n (3M + \sigma)^2}{2 \tau_x T}
		\\&
		-2\eta_y \frac{(k+3)^2}{9} \Ec{\sqn{\tilde{x}^{k+1} - x^k}}{\cF_k}
		+2\eta_y \frac{(k+2)^2}{9} \sqn{x^{k-1} - \tilde{x}^k}
		\\&
		+\frac{k+2}{3} \langle x^{k-1} - \tilde{x}^k, y^k - y \rangle
		-\frac{k+3}{3} \Ec{\langle x^k - \tilde{x}^{k+1}, y^{k+1} - y\rangle}{\cF_k}
		\\&
		+(\alpha_k^{-2} - \alpha_k^{-1}) \mQ(x, y, z, \bar{x}^k, \bar{y}^k, \bar{z}^k)
		-\alpha_k^{-2} \Ec{\mQ(x, y, z, \bar{x}^{k+1}, \bar{y}^{k+1}, \bar{z}^{k+1})}{\cF_k},
	\end{align*}
	where \annotate.

	Summing previously obtained inequalities for $k=0, \ldots, K-1$ and using $\E{\Ec{Z}{\cdot}} = \E{Z}$ we get
	\begin{align*}
		&\tau_x\alpha_{K}^{-2}\E{\sqn{x^{K} - x}}
		+\frac{1}{2\eta_y}\E{\sqn{y^{K} - y}}
		+\frac{1}{2\eta_z}\E{\sqn{\hat{z}^{K} - z}}
		+\frac{1}{2\eta_z}\E{\sqn{\eta_z^{K} m^{K}}_{\mP}}
		\\
		&\aleq{uses the definition of $\lambda_k$}
		\tau_x \alpha_0^{-2} \sqn{x^0 - x}
		+\frac{1}{2\eta_y}\sqn{y^0 - y}
		+\frac{1}{2\eta_z}\sqn{\hat{z}^0 - z}
		+\frac{1}{2\eta_z}\sqn{\eta_z^0 m^0}_{\mP}
		+\frac{n K (3M + \sigma)^2}{2 \tau_x T}\\
		&+\frac{8 \eta_y}{9} \sqn{\tilde{x}^0 - x^{-1}}
		+ \frac{2}{3} \langle x^{-1} - \tilde{x}^{0}, y^{0} - y\rangle\\
		&-2 \eta_y \frac{(K+2)^2}{9} \E{\sqn{\tilde{x}^K - x^{K-1}}}
		- \frac{K+2}{3} \E{\langle x^{K-1} - \tilde{x}^{K}, y^{K} - y\rangle}
		-\E{\sum_{k=1}^{K} \lambda_k \mQ(x, y, z, \bar{x}^k, \bar{y}^k, \bar{z}^k)}\\
		&\aeq{uses $\tilde{x}^0 = x^{-1}$}
		\tau_x \alpha_0^{-2} \sqn{x^0 - x}
		+\frac{1}{2\eta_y}\sqn{y^0 - y}
		+\frac{1}{2\eta_z}\sqn{\hat{z}^0 - z}
		+\frac{1}{2\eta_z}\sqn{\eta_z^0 m^0}_{\mP}
		+\frac{n K (3M + \sigma)^2}{2 \tau_x T}\\
		&-2 \eta_y \frac{(K+2)^2}{9} \E{\sqn{\tilde{x}^K - x^{K-1}}}
		- \frac{K+2}{3} \E{\langle x^{K-1} - \tilde{x}^{K}, y^{K} - y\rangle}
		-\E{\sum_{k=1}^{K} \lambda_k \mQ(x, y, z, \bar{x}^k, \bar{y}^k, \bar{z}^k)}\\
		&\aleq{uses the Cauchy-Schwarz inequality}
		\tau_x \alpha_0^{-2} \sqn{x^0 - x}
		+\frac{1}{2\eta_y}\sqn{y^0 - y}
		+\frac{1}{2\eta_z}\sqn{\hat{z}^0 - z}
		+\frac{1}{2\eta_z}\sqn{\eta_z^0 m^0}_{\mP}
		+\frac{n K (3M + \sigma)^2}{2 \tau_x T}\\
		&-2 \eta_y \frac{(K+2)^2}{9} \E{\sqn{\tilde{x}^K - x^{K-1}}}
		+\eta_y \frac{(K+2)^2}{9} \E{\sqn{\tilde{x}^K - x^{K-1}}}
		+\frac{1}{4 \eta_y} \E{\sqn{y^{K} - y}}\\
		&-\E{\sum_{k=1}^{K} \lambda_k \mQ(x, y, z, \bar{x}^k, \bar{y}^k, \bar{z}^k)}\\
		&=
		\tau_x \alpha_0^{-2} \sqn{x^0 - x}
		+\frac{1}{2\eta_y}\sqn{y^0 - y}
		+\frac{1}{2\eta_z}\sqn{\hat{z}^0 - z}
		+\frac{1}{2\eta_z}\sqn{\eta_z^0 m^0}_{\mP}
		+\frac{n K (3M + \sigma)^2}{2 \tau_x T}\\
		&-\eta_y \frac{(K+2)^2}{9} \E{\sqn{\tilde{x}^K - x^{K-1}}}
		+\frac{1}{4 \eta_y} \E{\sqn{y^{K} - y}}\\
		&-\E{\sum_{k=1}^{K} \lambda_k \mQ(x, y, z, \bar{x}^k, \bar{y}^k, \bar{z}^k)}\\
		&\aleq{uses~\ref{lemm:optimum}}
		\tau_x \alpha_0^{-2} \sqn{x^0 - x}
		+\frac{1}{2\eta_y}\sqn{y^0 - y}
		+\frac{1}{2\eta_z}\sqn{\hat{z}^0 - z}
		+\frac{1}{2\eta_z}\sqn{\eta_z^0 m^0}_{\mP}
		+\frac{n K (3M + \sigma)^2}{2 \tau_x T}\\
		&-\eta_y \frac{(K+2)^2}{9} \E{\sqn{\tilde{x}^K - x^{K-1}}}
		+\frac{1}{4 \eta_y} \E{\sqn{y^{K} - y}}\\
		&-\E{\sum_{k=1}^{K} \lambda_k \mQ(x, y, z, x_a^K, y_a^K, z_a^K)}\\
		&\aeq{uses the definitions of $\lambda_k$ and $\alpha_k$ and $\alpha_0 = 1$}
		\tau_x \sqn{x^0 - x}
		+\frac{1}{2\eta_y}\sqn{y^0 - y}
		+\frac{1}{2\eta_z}\sqn{\hat{z}^0 - z}
		+\frac{1}{2\eta_z}\sqn{\eta_z^0 m^0}_{\mP}
		+\frac{n K (3M + \sigma)^2}{2 \tau_x T}\\
		&-\eta_y \frac{(K+2)^2}{9} \E{\sqn{\tilde{x}^K - x^{K-1}}}
		+\frac{1}{4 \eta_y} \E{\sqn{y^{K} - y}}\\
		&-\E{\sum_{k=0}^{K-1} \alpha_k^{-1} \mQ(x, y, z, x_a^K, y_a^K, z_a^K)},
	\end{align*}
	where \annotate.

	Using the linearity of the expectation we get
	\begin{align*}
		&\E{\left( \sum_{k=0}^{K-1}\alpha_k^{-1} \right) \mQ(x, y, z, x_a^K, y_a^K, z_a^K)} \leq \frac{r}{3} \sqn{x} + \frac{6}{r} \sqn{y} + \frac{15 \chi^2}{r} \sqn{z} + \frac{3n(3M+\sigma)^2}{rKT}.
	\end{align*}

	Next, using the estimation \[\sum_{k=0}^{K-1} \alpha_k^{-1} \geq \frac{K^2}{6}\] and the fact that $x^0 = 0, y^0 = 0, z^0 = 0, m^0 = 0$ we obtain
	\begin{align*}
		&\E{\mQ(x, y, z, x_a^K, y_a^K, z_a^K)} \\&\leq \frac{1}{K^2} \left( 2r \sqn{x} + \frac{36}{r} \sqn{y} + \frac{90 \chi^2}{r} \sqn{z} \right) + \frac{18n(3M+\sigma)^2}{rKT},
	\end{align*}
	which concludes the proof.
\end{proof}

\section{Proof of Theorem~\ref{th:st_mntn}}\label{sec:proof_st_mntn}
This theorem is proved for the saddle point problems, as it directly implies the same convergence rate for the convex minimization problems.

We start by estimating the gap function defined in Definition~\ref{def:gap_spp}:
\begin{align*}
	&n \E{ p(\xi_o^K, \zeta^*) - p(\xi^*,\zeta_o^K)}\\
	&= \E{\sum_{i=1}^{n} f_i(\xi_o^K, \zeta^*) - f_i(\xi^*,\zeta_o^K)} +\frac{nr}{2}\E{ \sqn{\xi_o^K} - \sqn{\zeta^*} - \sqn{\xi^*} + \sqn{\zeta_o^K}}\\
	&=\E{\sum_{i=1}^{n} f_i(\xi_o^K, \zeta^*) - f_i(\xi^*,\zeta_o^K)} +\frac{nr}{2}\E{\sqn{x_o^K} - \sqn{x^*}}\\
	&\aleq{uses the convexity of squared norm}\E{\sum_{i=1}^{n} f_i(\xi_o^K, \zeta^*) - f_i(\xi^*,\zeta_o^K)} +\frac{r}{2}\E{\sqn{x_a^K} - \sqn{w^*}}\\
	&\aleq{uses the Assumption~\ref{ass:lipschitz}}\E{\sum_{i=1}^{n} f_i(\xi_{a, i}^K, \zeta^*) - f_i(\xi^*,\zeta_{a, i}^K) + \sqrt{2}M \norm{x_{a,i}^K - x_o^K}} +\frac{r}{2}\E{\sqn{x_a^K} - \sqn{w^*}}\\
	&\aleq{uses the definition of $F$}\E{F(\xi_a^K, \zeta^*) - F(\xi^*, \zeta_a^K)} +\frac{1}{2r_{yz}}\E{\sqn{x_a^K}} - \frac{1}{2r_{yz}}\E{\sqn{w^*}}\\&+ \E{\sum_{i=1}^{n}\sqrt{2}M \norm{x_{a,i}^K - x_o^K}}\\
	&\aeq{uses the definition of $\mD$}\E{\mD(x_a^K, x)} +\frac{1}{2r_{yz}}\E{\sqn{x_a^K}} - \frac{1}{2r_{yz}}\E{\sqn{w^*}} + \E{\sum_{i=1}^{n}\sqrt{2}M \norm{x_{a,i}^K - x_o^K}}\\
	&\aleq{uses the Cauchy-Schwarz inequality}\E{\mD(x_a^K, x)} +\frac{1}{2r_{yz}}\E{\sqn{x_a^K}} - \frac{1}{2r_{yz}}\E{\sqn{w^*}}\\&+ \E{\sqrt{\sum_{i=1}^n 2M^2}\sqrt{\sum_{i=1}^n \sqn{x_{a,i}^K - x_o^K}}}\\
	&\aeq{uses the definition of $\mP$}\E{\mD(x_a^K, x)} +\frac{1}{2r_{yz}}\E{\sqn{x_a^K}} - \frac{1}{2r_{yz}}\E{\sqn{w^*}} + \E{\sqrt{2n}M\norm{x_a^K}_{\mP}}\\
	&=\E{\mD(x_a^K, x)} +\frac{1}{2r_{yz}}\E{\sqn{x_a^K}} - \frac{1}{2r_{yz}}\E{\sqn{w^*}} + \sqrt{2n}M\E{\norm{x_a^K}_{\mP}},
\end{align*}
where \annotate.

Next, we take $y = -r_{yz}^{-1} x_a^K - z$. We also take $y^*$ and $z^*$ as in Lemma~\ref{lemm:optimum}. Then,
\begin{align*}
	&G(y, z) = \frac{r_{yz}}{2} \sqn{y+z} = \frac{1}{2 r_{yz}} \sqn{x_a^K};\\
	&G(y^*, z^*) = \frac{r_{yz}}{2} \sqn{y^* + z^*} \aeq{uses the proof of Lemma~\ref{lemm:optimum}} \frac{r_{yz}}{2} \sqn{(r_x - r) w^*} = \frac{1}{2r_{yz}} \sqn{w^*};\\
	&\langle y, x_a^K \rangle = - \frac{1}{r_{yz}} \sqn{x_a^K} - \langle z, x_a^K \rangle\\
	&\langle y^*, w^* \rangle = \sum_{i=1}^n \langle y^*_i, x^* \rangle \aeq{uses the definition of $y^*$} \sum_{i=1}^n (\langle \Delta_i^{\xi, *} + r_x \xi^*, \xi^* \rangle + \langle -\Delta_i^{\zeta, *} + r_x \zeta^*, \zeta^* \rangle )\\
	&=\sum_{i=1}^n (\langle \Delta_i^{\xi, *}, \xi^* \rangle - \langle \Delta_i^{\zeta, *}, \zeta^* \rangle + r_x \sqn{x^*} )\aeq{uses the definition of $x^*$ and~\eqref{eq:xstar}} \sum_{i=1}^n ( -r \sqn{x^*} + r_x \sqn{x^*} ) = -\frac{1}{r_{yz}} \sqn{w^*},
\end{align*}
where \annotate.

Hence,
\begin{align*}
	&\mQ(w^*, y, z, x_a^K, y^*, z^*) = \mD(x_a^K, w^*) - \langle y, x_a^K \rangle + \langle y^*, w^* \rangle - G(y, z) + G(y^*, z^*)\\
	&=\mD(x_a^K, w^*) + \frac{1}{r_{yz}} \sqn{x_a^K} + \langle z, x_a^K \rangle - \frac{1}{r_{yz}} \sqn{w^*} - \frac{1}{2 r_{yz}} \sqn{x_a^K} + \frac{1}{2r_{yz}} \sqn{w^*}\\
	&=\mD(x_a^K, w^*) + \frac{1}{2r_{yz}} \sqn{x_a^K} - \frac{1}{2r_{yz}} \sqn{w^*} + \langle z, x_a^K \rangle.
\end{align*}

Combining this with previously obtained inequality we get
\begin{align*}
	&n \E{p(\xi_o^K, \zeta^*) - p(\xi^*,\zeta_o^K)}\\
	&\leq \E{\mQ(w^*, y, z, x_a^K, y^*, z^*)} - \E{\langle z, x_a^K \rangle} + \E{\sqrt{2n}M\norm{x_a^K}_{\mP}}\\
	&\aleq{uses the convexity of $\mQ$ in $y_o$ and $z_o$ and Lemma~\ref{lemm:optimum}}\E{\mQ(w^*, y, z, x_a^K, y_a^K, z_a^K)} - \E{\langle z, x_a^K \rangle} + \E{\sqrt{2n}M\norm{x_a^K}_{\mP}}.
\end{align*}
where \annotate.

Now, we choose $z \in \mathcal{L}^\perp$ as follows:
\begin{equation}
	z = \begin{cases}
		\sqrt{2n}M\norm{\mP x_a^K}^{-1}\mP x_a^K & \text{if } x_a^K \neq 0\\
		0 & \text{if } x_a^K = 0
	\end{cases}.
\end{equation}

Hence,
\begin{align*}
	&n \E{p(\xi_o^K, \zeta^*) - p(\xi^*,\zeta_o^K)} \leq\E{\mQ(w^*, y, z, x_a^K, y_a^K, z_a^K)}\\
	&\aleq{uses Lemma~\ref{lemm:gap_iteration}}\frac{1}{K^2} \left( 2r \sqn{w^*} + \frac{36}{r} \sqn{y} + \frac{90 \chi^2}{r} \sqn{z} \right) + \frac{18n(3M+\sigma)^2}{rKT}\\
	&\aeq{uses the definition of $y$}\frac{1}{K^2} \left( 2r \sqn{w^*} + \frac{36}{r} \sqn{r_{yz}^{-1} x_a^K + z} + \frac{90 \chi^2}{r} \sqn{z} \right) + \frac{18n(3M+\sigma)^2}{rKT}\\
	&=\frac{1}{K^2} \left( 2r \sqn{w^*} + \frac{36}{r} \sqn{r_{yz}^{-1} (x_a^K - w^* + w^*) + z} + \frac{90 \chi^2}{r} \sqn{z} \right) + \frac{18n(3M+\sigma)^2}{rKT}\\
	&\aleq{uses the parallelogram rule}\frac{1}{K^2} \left( 2r \sqn{w^*} + \frac{108}{r r_{yz}^2} \sqn{x_a^K - w^*} + \frac{108}{r r_{yz}^2} \sqn{w^*} + \frac{108}{r} \sqn{z} + \frac{90 \chi^2}{r} \sqn{z} \right)\\&+ \frac{18n(3M+\sigma)^2}{rKT}\\
	&\leq\frac{1}{K^2} \left( 2r \sqn{w^*} + \frac{108}{r r_{yz}^2} \sqn{x_a^K - w^*} + \frac{108}{r r_{yz}^2} \sqn{w^*} + \frac{198 \chi^2}{r} \sqn{z} \right)\\&+ \frac{18n(3M+\sigma)^2}{rKT}\\
	&\aleq{uses the definition of $z$}\frac{1}{K^2} \left( 2r \sqn{w^*} + \frac{108}{r r_{yz}^2} \sqn{x_a^K - w^*} + \frac{108}{r r_{yz}^2} \sqn{w^*} + \frac{198 n \chi^2 M^2}{r} \right)\\&+ \frac{18n(3M+\sigma)^2}{rKT}\\
	&\aleq{uses Lemma~\ref{lemm:optimum}}\frac{1}{K^2} \left( \frac{4nM^2}{r}  + \frac{216nM^2}{r^3 r_{yz}^2} + \frac{108}{r r_{yz}^2} \sqn{x_a^K - w^*} + \frac{198 n \chi^2 M^2}{r} \right)\\&+ \frac{18n(3M+\sigma)^2}{rKT}\\
	&\aeq{uses the definition of $r_{yz}$}\frac{1}{K^2} \left( \frac{28nM^2}{r} + \frac{108}{r r_{yz}^2} \sqn{x_a^K - w^*} + \frac{198 n \chi^2 M^2}{r} \right)\\&+ \frac{18n(3M+\sigma)^2}{rKT}\\
	&\leq \frac{226 n \chi^2 M^2}{r K^2} + \frac{18n(3M+\sigma)^2}{rKT} + \frac{12r}{K^2} \sqn{x_a^K - w^*},
\end{align*}
where \annotate.

To estimate $r\sqn{x_a^K - w^*}$ we have
\begin{align*}
	&\frac{r_x}{2} \sqn{x_a^K - w^*} \aleq{uses strong convexity of $\mQ$ in $x_o$ and Lemma~\ref{lemm:optimum}} \mQ(w^*, y^*, z^*, x_a^K, y^*, z^*)\\
	&\aleq{uses Lemma~\ref{lemm:optimum}}\mQ(w^*, y^*, z^*, x_a^K, y_a^K, z_a^K)\\
	&\aleq{uses Lemma~\ref{lemm:gap_iteration}}\frac{1}{K^2} \left( 2r \sqn{x^*} + \frac{36}{r} \sqn{y^*} + \frac{90 \chi^2}{r} \sqn{z^*} \right) + \frac{18n(3M+\sigma)^2}{rKT}\\
	&\aleq{uses Lemma~\ref{lemm:optimum}}\frac{1}{K^2} \left( \frac{4nM^2}{r} + \frac{72 (1 + r_x/r)^2 n M^2}{r} + \frac{720 n \chi^2 M^2}{r} \right) + \frac{18n(3M+\sigma)^2}{rKT},
\end{align*}
where \annotate.

Using the definition of $r_x$, we get
\begin{align*}
	&r \sqn{x_a^K - w^*} \leq \frac{3}{K^2} \left( \frac{204nM^2}{r} + \frac{720 n \chi^2 M^2}{r} \right) + \frac{18n(3M+\sigma)^2}{rKT}\\
	&\leq \frac{2772 n \chi^2 M^2}{rK^2} + \frac{18n(3M+\sigma)^2}{rKT}.
\end{align*}

Combining and dividing by $n$, we obtain\begin{align*}
	&\E{p(\xi_o^K, \zeta^*) - p(\xi^*,\zeta_o^K)}\\
	&\leq \frac{226 n \chi^2 M^2}{r K^2} + \frac{18n(3M+\sigma)^2}{rKT} + \frac{12}{K^2} \left(\frac{2772 n \chi^2 M^2}{rK^2} + \frac{18n(3M+\sigma)^2}{rKT} \right).\end{align*}

Now, taking \begin{equation*}
	K \geq \mathcal{O}\left(\frac{\chi M}{\sqrt{r \varepsilon}}\right) \quad \text{and} \quad KT \geq \mathcal{O}\left(\frac{(M + \sigma)^2}{r \varepsilon}\right)
\end{equation*}
we achieve $G_{\mathrm{SPP}}(x_o^K) \leq \varepsilon,$ which concludes the proof for the saddle point problems.

To see that the obtained upper bound also holds for convex problems, observe that any convex minimization problem can be cast as a special case of a saddle-point problem. Specifically, consider a convex optimization problem of the form
\begin{equation}
	\min_{x \in \R^d} p(x) \label{eq:special_case_cvx}.
\end{equation}

This problem can be equivalently rewritten as the saddle-point problem
\begin{equation}
	\min_{x \in \R^d} \max_{y \in \mathbb{R}} \left\{ p(x) + \langle y, 0 \rangle \right\} \label{eq:special_case}.
\end{equation}
Therefore, ensuring $G_{\mathrm{SPP}}(x_o^K) \leq \varepsilon$ for the problem~\eqref{eq:special_case} implies $G_{\mathrm{CVX}}(x_o^K) \leq \varepsilon$ for the problem \eqref{eq:special_case_cvx}, which concludes the proof.

\section{Proof of Theorem~\ref{th:mntn}}\label{sec:proof_mntn}
We have the problem of the form
\begin{equation}\label{eq:non_regularized}
	\min_{\xi \in \R^{d_\xi}} \max_{\zeta \in \R^{d_\zeta}} f(\xi, \zeta) = \frac{1}{n} \sum_{i=1}^n f_i(\xi, \zeta).
\end{equation}

Let $x^* = (\xi^*, \zeta^*)$ be the solution to the problem \eqref{eq:non_regularized}, which exists due to Assumption~\ref{ass:radius} and $\norm{x^*} \leq R$. Consider the regularized version by taking
\begin{equation}\label{eq:regularized}
	p(\xi, \zeta) = f(\xi, \zeta) + \frac{r}{2}\sqn{\xi} - \frac{r}{2}\sqn{\zeta}.
\end{equation}

Let $x^*_r = (\xi^*_r, \zeta^*_r)$ be the solution to the problem \eqref{eq:regularized}, which always exists and unique.

To achieve $\E{p(\xi_o^K, \zeta^*_r) - p(\xi^*_r, \zeta_o^K)} \leq \varepsilon$ we require $\cO\left(\frac{\chi M}{\sqrt{r \varepsilon}}\right)$ decentralized communications and $\cO\left(\frac{(M + \sigma)^2}{r \varepsilon}\right)$ oracle calls.

Then, we estimate the saddle-point gap for the problem \eqref{eq:non_regularized}:
\begin{align*}
	&\E{f(\xi_o^K, \zeta^*) - f(\xi^*, \zeta_o^K)}\\
	&= \E{p(\xi_o^K, \zeta^*) - p(\xi^*, \zeta_o^K) - \frac{r}{2}\sqn{\xi_o^K} - \frac{r}{2}\sqn{\zeta_o^K}} + \frac{r}{2}\sqn{\zeta^*}+ \frac{r}{2}\sqn{\xi^*}\\
	&\aleq{uses the definition of $x_r^*$ and $x^*$}\E{p(\xi_o^K, \zeta_r^*) - p(\xi_r^*, \zeta_o^K)} + \frac{r}{2}\sqn{x^*} \aleq{uses the Assumption~\ref{ass:radius}} \varepsilon + \frac{rR^2}{2},
\end{align*}
where \annotate.

Then, taking $r = \frac{\varepsilon}{R^2}$, we achieve $\E{f(\xi_o^K, \zeta^*) - f(\xi^*, \zeta_o^K)} \leq 2\varepsilon$.

Thus, we require
\begin{equation}
	\cO\left(\frac{\chi MR}{ \varepsilon}\right) \text{ decentralized communications}
\end{equation}
and
\begin{equation}
	\cO\left(\frac{(M + \sigma)^2R^2}{ \varepsilon^2}\right) \text{ oracle calls},
\end{equation}
which concludes the proof for the saddle point problems. For the convex problems the approach is the same as in the proof of Theorem~\ref{th:st_mntn}.

\section{Proof of Corollary~\ref{cor:asym}}\label{sec:proof_asym}
We start by rescaling
\begin{align*}
	\xi \rightarrow \frac{\xi}{\sqrt{r_\xi}},\quad \zeta \rightarrow \frac{\zeta}{\sqrt{r_\zeta}}.
\end{align*}
Also, define functions
\begin{align*}
	\tilde{f}_i(\xi, \zeta) = f_i \left(\frac{\xi}{\sqrt{r_\xi}}, \frac{\zeta}{\sqrt{r_\zeta}} \right).
\end{align*}
Hence, the source function becomes
\begin{align*}
	\tilde{p}(\xi, \zeta) &= p\left(\frac{\xi}{\sqrt{r_\xi}}, \frac{\zeta}{\sqrt{r_\zeta}} \right)\\
	&= \frac{1}{n} \sum_{i=1}^n f_i \left(\frac{\xi}{\sqrt{r_\xi}}, \frac{\zeta}{\sqrt{r_\zeta}} \right) + \frac{r_\xi}{2} \sqn{\frac{\xi}{\sqrt{r_\xi}}} - \frac{r_\zeta}{2} \sqn{\frac{\zeta}{\sqrt{r_\zeta}}}\\
	&= \frac{1}{n} \sum_{i=1}^n f_i \left(\frac{\xi}{\sqrt{r_\xi}}, \frac{\zeta}{\sqrt{r_\zeta}} \right) + \frac{1}{2} \sqn{\xi} - \frac{1}{2} \sqn{\zeta}\\
	&= \frac{1}{n} \sum_{i=1}^n \tilde{f_i} \left( \xi, \zeta \right) + \frac{1}{2} \sqn{\xi} - \frac{1}{2} \sqn{\zeta}.
\end{align*}

This function has symmetric convexity and concavity constants and thus can be solved as the problem \eqref{prob:spp}. From Assumption~\ref{ass:lipschitz} we know that
\begin{align*}
	\norm{\partial f_i(\xi, \zeta)} \leq M.
\end{align*}
Hence, for the scaled problem we have
\begin{align*}
	\norm{\partial \tilde{f_i}(\xi, \zeta)} &= \norm{\begin{pmatrix}
			\partial_\xi \tilde{f_i}(\xi, \zeta)\\
			\partial_\zeta \tilde{f_i}(\xi, \zeta)	
	\end{pmatrix}} = \norm{\begin{pmatrix}
			\frac{1}{\sqrt{r_\xi}}\partial_{\frac{\xi}{\sqrt{r_\xi}}} f_i \left(\frac{\xi}{\sqrt{r_\xi}}, \frac{\zeta}{\sqrt{r_\zeta}} \right)\\
			\frac{1}{\sqrt{r_\zeta}}\partial_{\frac{\xi}{\sqrt{r_\zeta}}} f_i \left(\frac{\xi}{\sqrt{r_\xi}}, \frac{\zeta}{\sqrt{r_\zeta}} \right)	
	\end{pmatrix}}\\
	&\leq \sqrt{\sqn{\frac{1}{\sqrt{r_\xi}} M} + \sqn{\frac{1}{\sqrt{r_\zeta}} M}} = M \sqrt{\frac{1}{r_\xi} + \frac{1}{r_\zeta}}.
\end{align*}

Thus, $\tilde{M} = M \sqrt{\frac{1}{r_\xi} + \frac{1}{r_\zeta}}.$ Similarly,
\begin{align*}
	\tilde{\sigma} = \sigma \sqrt{\frac{1}{r_\xi} + \frac{1}{r_\zeta}}.
\end{align*}

From Theorem~\ref{th:st_mntn} we get that solving problem
\begin{align*}
	\min_{\xi \in \mathbb{R}^{d_\xi}}\max_{\zeta \in \mathbb{R}^{d_\zeta}} \left[ \frac{1}{n} \sum_{i=1}^n \tilde{f_i} \left( \xi, \zeta \right) + \frac{1}{2} \sqn{\xi} - \frac{1}{2} \sqn{\zeta} \right]
\end{align*}
requires
\begin{align*}
	\cO\left(\frac{\chi M}{\sqrt{r \varepsilon}}\right) \text{ decentralized communications}
\end{align*}
and
\begin{align*}
	\cO \left(\frac{(M + \sigma)^2}{r \varepsilon}\right) \text{ oracle calls}
\end{align*}
to achieve $G_{\mathrm{SPP}}(x_o^K) \leq \varepsilon$. Thus, the asymmetric problem \eqref{prob:spp_asymmetric} can be solved in
\begin{align*}
	\cO\left(\frac{\chi M}{\sqrt{\varepsilon}} \sqrt{\frac{1}{r_\xi} + \frac{1}{r_\zeta}} \right) \text{ decentralized communications}
\end{align*}
and
\begin{align*}
	\cO \left(\frac{(M + \sigma)^2 }{\varepsilon} \left(\frac{1}{r_\xi} + \frac{1}{r_\zeta}\right) \right) \text{ oracle calls},
\end{align*}
which concludes the proof.

\end{document}